\documentclass{amsart}
\usepackage[all,2cell,ps]{xy}
\usepackage{amssymb}
\usepackage{amsmath}

\def\latex/{{\protect\LaTeX}}
\def\latexe/{{\protect\LaTeXe}}
\def\amslatex/{{\protect\AmS-\protect\LaTeX}}
\def\tex/{{\protect\TeX}}
\def\amstex/{{\protect\AmS-\protect\TeX}}
\def\bibtex/{{Bib\protect\TeX}}
\def\makeindx/{\textit{MakeIndex}}

\newtheorem{theorem}{Theorem}[section]

\newtheorem{corollary}[theorem]{Corollary}
\newtheorem{prop}[theorem]{Proposition}

\theoremstyle{definition}

\newtheorem{example}[theorem]{Example}

\newtheorem{remark}[theorem]{Remark}
\newtheorem{question}[theorem]{Question}
\numberwithin{equation}{section}

\newcommand{\ZZ}{\mathbb{Z}}      % for Integers

\begin{document}

\title{VANISHING OF TOR OVER COMPLETE INTERSECTIONS}
\author{Olgur Celikbas}
\thanks{This work will form part of the author's Ph.D. dissertation at the University Of Nebraska, under the direction of Professor Roger Wiegand and Professor Mark Walker.}
\date{September 28, 2009}
\address{Department of Mathematics, University Of Nebraska, Lincoln, Ne 68588-0323 USA }
\email{s-ocelikb1@math.unl.edu}



\begin{abstract}
In this paper we are concerned with the vanishing of $\textnormal{Tor}$ over
complete intersection rings. Building on results of C.\ Huneke, D.\
Jorgensen and R.\ Wiegand, and, more recently, H.\ Dao, we obtain new
results showing that good depth properties on the $R$-modules $M$, $N$ and
$M\otimes_RN$ force the vanishing of $\textnormal{Tor}^{R}_{i}(M,N)$ for all $i\geq 1$.
\end{abstract}

\maketitle
\section{Introduction}
Let $(R,\mathfrak{m})$ be a \emph{local} ring, i.e., a commutative Noetherian ring with unique maximal ideal $\mathfrak{m}$, and let $M$ be a finitely generated $R$-module.

The \emph{codimension} of $R$ is defined to be the nonnegative integer $\textnormal{embdim}(R)-\textnormal{dim}(R)$ where $\textnormal{embdim}(R)$, the \emph{embedding} dimension of $R$, is the minimal number of generators of $\mathfrak{m}$. Let $\hat{R}$ denote the $m$-adic completion of $R$. Recall that $R$ is said to be a \emph{complete intersection} when $\hat{R}$ is of the form $S/(\underline{f})$ where $(S,\mathfrak{n})$ is a complete regular local ring and $\underline{f}$ is a regular sequence of $S$ contained in $\mathfrak{n}$. Since $S/(g)$ is again a regular local ring if $g\in \mathfrak{n}-\mathfrak{n}^{2}$, we can always assume, by shortening the sequence if necessary, that $(\underline{f})\subseteq \mathfrak{n}^{2}$, and in this case the codimension of $R$ is equal to the length of the regular sequence $\underline{f}$.

The depth of $M$, denoted by $\textnormal{depth}_{R}(M)$, is the length of a maximal $M$-regular sequence contained in $\mathfrak{m}$. (The depth of the zero module is defined to be $\infty$.) We say that $M$ is \emph{Cohen-Macaulay} if $M=0$, or $M\neq 0$ and $\textnormal{depth}_{R}(M)=\textnormal{dim}_{R}(M)$. $M$ is said to be \emph{maximal Cohen-Macaulay} if $M$ is a nonzero Cohen-Macaulay module and $\textnormal{depth}_{R}(M)=\textnormal{dim}(R)$ (cf. \cite{BH}).

We set $X^{n}(R)=\{q\in \text{Spec}(R): \text{dim}(R_{q}) \leq n \}$ and say $M$ is free on $X^{n}(R)$ if $M_{q}$ is a free $R_{q}$-module for all $q \in X^{n}(R)$. Following \cite{HW1}, we say $M$ is \emph{free of constant rank} on $X^{n}(R)$ if there exists an $s$ such that $M_{q}\cong R^{(s)}_{q}$ for all $q\in X^{n}(R)$. If $M$ is free of constant rank on $X^{0}(R)$, then we say $M$ has \emph{constant rank}. We define a \emph{vector bundle} over $R$ to be an $R$-module which is free on $X^{d-1}(R)$, where $d=\text{dim}(R)$.

Let $S$ be the set of non-zerodivisors of $R$, and let $K=S^{-1}R$ be the total quotient ring of $R$. Then the torsion submodule of $M$, $t(M)$, is the kernel of the natural map $M\rightarrow M\otimes_{R}K$. $M$ is called \emph{torsion} provided $t(M)=M$, and \emph{torsion-free} provided $t(M)=0$.

For an integer $n\geq 0$, we say $M$ satisfies $(S_{n})$ if $\text{depth}_{R_{q}}(M_{q})\geq \text{min}\left\{n, \textnormal{dim}(R_{q})\right\}$ for all $q\in \textnormal{Spec}(R)$ (cf. \cite{EG}). (Note that this definition is different from the one given in \cite[5.7.2.I]{EGA}.) If $R$ is Cohen-Macaulay, then $M$ satisfies $(S_{n})$ if and only if every $R$-regular sequence $x_{1},x_{2},\dots, x_{k}$, with $k\leq n$, is also an $M$-regular sequence \cite{Sam}. In particular, if $R$ is Cohen-Macaulay, then $M$ satisfies $(S_{1})$ if and only if it is torsion-free. Moreover, if $R$ is Gorenstein, then $M$ satisfies $(S_{2})$ if and only if it is reflexive, i.e., the natural map $M \rightarrow M^{\ast\ast}$ is bijective, where $M^{\ast}=\textnormal{Hom}_{R}(M,R)$ (see \cite[3.6]{EG}).

If $\textbf{F}:\ldots \rightarrow F_{2} \rightarrow F_{1} \rightarrow F_{0}  \rightarrow 0$ is a minimal free resolution of $M$ over $R$, then the rank of $F_{n}$, denoted by $\beta^{R}_{n}(M)$, is the nth \emph{Betti} number of $M$. The nth \emph{syzygy} of $M$, denoted by $\textnormal{syz}^{R}_{n}(M)$, is the image of the map $F_{n} \rightarrow F_{n-1}$ and is unique up to isomorphism (Set $\textnormal{syz}^{R}_{0}(M)=M$.)

The module $M$ has \emph{complexity} $s$ \cite[3.1]{Av1}, written as $\textnormal{cx}_{R}(M)=s$, provided $s$ is the least nonnegative integer for which there exists a real number $\gamma$ such that $\beta^{R}_{n}(M)\leq \gamma \cdot n^{s-1}$ for all $n\gg 0$. It may be that no such $s$ and $\gamma$ exist (e.g. \cite[4.2.2]{Av2}), in which case we set $\textnormal{cx}_{R}(M)=\infty$. If $R$ is a complete intersection, then the complexity of $M$ is less than or equal to the codimension of $R$ (cf. \cite{Gu}). Moreover, over a complete intersection of codimension $c$, there exist modules of complexity $r$ for any non-negative integer $r \leq c$ (cf. \cite[6.6]{Av1} or \cite[3.1-3.3]{AGP}). It follows from the definition that $M$ has finite projective dimension if and only if $\textnormal{cx}_{R}(M)=0$, and has bounded Betti numbers if and only if $\textnormal{cx}_{R}(M)\leq 1$.

In this paper, our main goal is to examine certain conditions on the finitely generated $R$-modules $M$, $N$ and $M\otimes_{R}N$ that imply the vanishing of homology modules $\textnormal{Tor}^{R}_{i}(M,N)$ when $R$ is a local complete intersection. Our motivation comes from theorems of  Huneke-Jorgensen-Wiegand \cite{HJW} and H. Dao \cite{Da2} (Dao's theorem is stated as Theorem $3.3$ below.) In section $3$, we prove the following theorem as Corollary $3.12$.

\begin{theorem} Let $(R,\mathfrak m)$ be a local complete intersection, and let $M$ and $N$ be finitely generated $R$-modules. Assume $M$, $N$ and $M\otimes_{R}N$ are maximal Cohen-Macaulay. Set $r=\textnormal{min}\left\{\textnormal{cx}_{R}(M),\textnormal{cx}_{R}(N) \right\}$.
\begin{enumerate}
\item If $M$ is free on $X^{r}(R)$, then $\textnormal{Tor}^{R}_{i}(M,N)=0$ for all $i\geq 1$.
\item Assume $M$ is free on $X^{r-1}(R)$. Then $\textnormal{Tor}^{R}_{i}(M,N)=0$ for all even $i\geq 2$. Moreover, if $\textnormal{Tor}^{R}_{j}(M,N)=0$ for some odd $j\geq 1$, then $\textnormal{Tor}^{R}_{i}(M,N)=0$ for all $i\geq 1$.
\end{enumerate}
\end{theorem}

Examples $3.13$ and $3.14$ show that the assumptions that $M$ is free on $X^{r-1}(R)$, respectively, $X^{r}(R)$ are essential for Theorem $1.1$.

Section $4$ contains some further applications about tensor products of modules. An example of the results in section $4$ is the following theorem which is proved as Theorem $4.15$.

\begin{theorem} Let $(R,\mathfrak m)$ be a local complete intersection, and let $M$ and $N$ be finitely generated $R$-modules, at least one of which has constant rank. Assume $M$, $N$ and $M\otimes_{R}N$ are maximal Cohen-Macaulay. Set $r=\textnormal{max}\left\{\textnormal{cx}_{R}(M),\textnormal{cx}_{R}(N) \right\}$.
If $\textnormal{Tor}^{R}_{1}(M,N)=\textnormal{Tor}^{R}_{2}(M,N)=\ldots =\textnormal{Tor}^{R}_{r-1}(M,N)=0$, then $\textnormal{Tor}^{R}_{i}(M,N)=0$ for all $i\geq 1$.
\end{theorem}

\section{Preliminary Results}
In this section, for the reader's convenience, we record some of the major theorems about the vanishing of Tor that will be used throughout the paper.

The rigidity of Tor starts with the following famous theorem of Auslander and Lichtenbaum:

\begin{theorem} (\cite[Corollary 2.2]{Au} and \cite[Corollary 1]{Li}) Let $(R,\mathfrak m)$ be a regular local ring, and let $M$ and $N$ be finitely generated $R$-modules. If $\textnormal{Tor}^{R}_{n}(M,N)=0$ for some $n\geq 1$, then $\textnormal{Tor}^{R}_{i}(M,N)=0$ for all $i\geq n$.
\end{theorem}

The above result was first proved by Auslander \cite{Au} for unramified regular local rings, and then extended to all regular local rings by Lichtenbaum in \cite{Li}, where the ramified case was proved. Murthy \cite{Mu} proved that a similar rigidity theorem holds over an arbitrary complete intersection of codimension $c$, provided one assumes the vanishing of $c+1$ consecutive Tor modules:

\begin{theorem} (\cite[1.9]{Mu}) Let $(R,\mathfrak m)$ be a local complete intersection of codimension $c$, and let $M$ and $N$ be finitely generated $R$-modules. If $$\textnormal{Tor}^{R}_{n}(M,N)=
\textnormal{Tor}^{R}_{n+1}(M,N)=\dots=
\textnormal{Tor}^{R}_{n+c}(M,N)=0$$ for some $n\geq 1$, then $\textnormal{Tor}^{R}_{i}(M,N)=0$ for all $i\geq n$.
\end{theorem}

Later Huneke and Wiegand \cite[2.4]{HW1} proved that, over a hypersurface (a complete intersection of codimension one), the vanishing interval in Murthy's theorem may be reduced by one under certain length and dimension restrictions. This result was then extended to arbitrary complete intersections by using an induction argument in the following form:

\begin{theorem} (\cite[1.9]{HJW}) Let $(R,\mathfrak m)$ be a $d$-dimensional local ring such that $\hat{R}=S/(\underline{f})$ where $(S,\mathfrak{n})$ is a complete regular local ring and $\underline{f}=f_{1},f_{2},\dots,f_{c}$, for $c\geq 1$, is a regular sequence of $S$   contained in $\mathfrak{n}^{2}$. Let $M$ and $N$ be finitely generated $R$-modules. Assume the following conditions hold:
\begin{enumerate}
    \item $M\otimes_{R}N$ has finite length.
    \item $\textnormal{dim}(M)+\textnormal{dim}(N)<d+c$.
    \item $\textnormal{Tor}^{R}_{n}(M,N)=\textnormal{Tor}^{R}_{n+1}(M,N)=\ldots =\textnormal{Tor}^{R}_{n+c-1}(M,N)=0$ for some positive integer $n$.
    \item Either $n>d$ or else $S$ is unramified.
\end{enumerate}
Then $\textnormal{Tor}^{R}_{i}(M,N)=0$ for all $i\geq n$.
\end{theorem}

We will use several results of D. Jorgensen. The next one we record is a generalization of Murthy's theorem \cite[1.9]{Mu}.

\begin{theorem} (\cite[2.3]{Jo1}) Let $(R,\mathfrak m)$ be a local complete intersection of dimension $d$, and let $M$ and $N$ be finitely generated $R$-modules. Set $r=\textnormal{min}\left\{\textnormal{cx}_{R}(M),\textnormal{cx}_{R}(N)\right\}$ and $b=\textnormal{max}\left\{\textnormal{depth}_{R}(M), \textnormal{depth}_{R}(N)\right\}$. If $$\textnormal{Tor}^{R}_{n}(M,N)=\textnormal{Tor}^{R}_{n+1}(M,N)= \ldots =\textnormal{Tor}^{R}_{n+r}(M,N)=0$$ for some $n\geq d-b+1$, then $\textnormal{Tor}^{R}_{i}(M,N)=0$ for all $i\geq d-b+1$.
\end{theorem}

\begin{theorem}(\cite[1.3]{Jo1}) Let $(R,\mathfrak m)$ be a local complete intersection of codimension $c\geq 1$, and let $F$ be a finite set of $R$-modules. Assume $R$ is complete and has infinite residue field. Then there exists
a complete intersection $R_{1}$ of codimension $c-1$ and a non-zerodivisor $x$ of $R_{1}$ such that
$R=R_{1}/(x)$ and, for all $M\in F$,
\begin{equation*} \textnormal{cx}_{R_{1}}(M)=
\begin{cases}
\textnormal{cx}_{R}(M)-1, & \textnormal{ if } \textnormal{cx}_{R}(M)>0 \\ 0, & \textnormal{ if } \textnormal{cx}_{R}(M)=0.    \\
\end{cases}
\end{equation*}
\end{theorem}

As stated in \cite{Jo1}, Theorem $2.5$ also follows from a theorem of Avramov (cf. \cite[3.2.3 \text{and} 3.6]{Av1}).

\begin{theorem}(\cite[2.7]{Jo1}) Let $(R,\mathfrak m)$ be a $d$-dimensional local complete intersection, and let $M$ and $N$ be finitely generated $R$-modules, at least one of which has complexity one. Set $b=\textnormal{max}\left\{\textnormal{depth}_{R}(M), \textnormal{depth}_{R}(N)\right\}$. Then $\textnormal{Tor}^{R}_{i}(M,N)\cong \textnormal{Tor}^{R}_{i+2}(M,N)$ for all $i\geq d-b+1$.
\end{theorem}

Another important result that we will frequently use throughout the paper is the depth formula. Auslander \cite[1.2]{Au} proved that if $(R,\mathfrak m)$ is a local ring, $M$ and $N$ are finitely generated $R$-modules such that $M$ has finite projective dimension and $q=\text{sup}\{i:\textnormal{Tor}^{R}_{i}(M,N)\neq 0\}$, then the equality $$\textnormal{depth}(M)+\textnormal{depth}(N)=\textnormal{depth}(R)+ \textnormal{depth}(\textnormal{Tor}^{R}_{q}(M,N))-q$$
holds, provided either $q=0$ or $\textnormal{depth}(\textnormal{Tor}^{R}_{q}(M,N))\leq 1$.
We refer the above equality as \emph{Auslander's depth formula}. This remarkable equality, for the case where $q=0$, was later obtained by Huneke and Wiegand for complete intersections without the finite projective dimension restriction on $M$ (cf. also \cite{ArY} and \cite{I}).

\begin{theorem} \cite[2.5]{HW1} Let $(R,\mathfrak m)$ be a local complete intersection, and let $M$ and $N$ be finitely generated $R$-modules. If $\textnormal{Tor}^{R}_{i}(M,N)=0$ for all $i\geq 1$, then the depth formula for $M$ and $N$ holds: $$\textnormal{depth}(M)+\textnormal{depth}(N)=\textnormal{depth}(R)+\textnormal{depth}(M\otimes_{R}N)$$
\end{theorem}

Most of the applications in this paper will rely on the following result of H. Dao.

\begin{theorem} (\cite[7.7]{Da2}) Let $(R,\mathfrak m)$ be an admissible local complete intersection (i.e., $\hat{R}$ is the quotient, by a regular sequence, of a power series ring over a field or a discrete valuation ring) of codimension $c\geq 1$, and let $M$ and $N$ be finitely generated $R$-modules. Assume the following conditions hold:
\begin{enumerate}
    \item $\textnormal{Tor}^{R}_{1}(M,N)=\textnormal{Tor}^{R}_{2}(M,N)=\ldots =\textnormal{Tor}^{R}_{c}(M,N)=0$.
    \item $\textnormal{depth}(N)>0$ and $\textnormal{depth}(M\otimes_{R}N)>0$.
    \item $\textnormal{Tor}^{R}_{i}(M,N)$ has finite length for all $i \gg 0$.
\end{enumerate}
Then $\textnormal{Tor}^{R}_{i}(M,N)=0$ for all $i\geq 1$.
\end{theorem}

\section{Proof of Theorem $1.1$}
In this section, we will prove a more general version of Theorem $1.1$ described in the introduction. Our main instruments will be \emph{pushforwards} and \emph{quasi-liftings} (cf. \cite{HJW} and \cite{EG}). First we recall their definitions:
\vspace{0.1in}

Let $R$ be a Gorenstein ring, $M$ a finitely generated torsion-free $R$-module, and $\{f_{1},f_{2},\dots, f_{m}\}$ a minimal generating set for $M^{\ast}$. Let  $\displaystyle{\delta: R^{(m)} \twoheadrightarrow M^{\ast}}$ be defined by $\delta(e_{i})=f_{i}$ for $i=1,2,\dots, m$ where $\{e_{1},e_{2},\dots, e_{m}\}$ is the standard basis for $R^{(m)}$. Then, composing the natural map $M\hookrightarrow M^{\ast\ast}$ with $\delta^{\ast}$, we obtain a short exact sequence
$$\;\;\textnormal{(PF)}\;\;\; 0 \rightarrow M \stackrel{u}{\rightarrow} R^{(m)} \rightarrow M_{1} \rightarrow 0$$ where $u(x)=(f_{1}(x),f_{2}(x),\dots, f_{m}(x))$ for all $x\in M$.
Any module $M_{1}$ obtained in this way is referred to as a \emph{pushforward} of $M$. We should note that such a construction is unique, up to a non-canonical isomorphism (cf. \textnormal{page} 62 of \cite{EG}). Indeed, suppose $\{g_{1},g_{2},\dots, g_{m}\}$ is another minimal generating set for $M^{\ast}$. Then, by the uniqueness of minimal resolutions, there exists an isomorphism $\varphi$ so that the following diagram commutes,
$$
\xymatrixrowsep{0.4in}\xymatrixcolsep{0.4in}\xymatrix{
R^{(m)} \ar@{.>}[d]^{\cong}_{\varphi} \ar[r]^{\delta} & M^{\ast}\ar@{=}[d] \ar[r]^{}  &   0\\
R^{(m)} \ar[r]_{\chi} & M^{\ast} \ar[r]^{}  &   0}
$$
where $\chi(e_{i})=g_{i}$ for $i=1,2,\dots, m$. It follows that $\varphi^{t}v=u$ where $\varphi^{t}$ is the transpose of $\varphi$ and $v(x)=(g_{1}(x),g_{2}(x),\dots, g_{m}(x))$. Hence we have the following commutative diagram:
$$
\xymatrixrowsep{0.4in}\xymatrixcolsep{0.4in}\xymatrix{
0 \ar[r]^{} & M \ar@{=}[d] \ar[r]^{v} &  R^{(m)} \ar[d]^{\cong}_{\varphi^{t}} \ar[r]^{} & M_{1}^{'}\ar[d]^{\cong} \ar[r]^{}  &   0\\
0\ar[r]^{} & M \ar[r]^{u} & R^{(m)} \ar[r]_{} & M_{1} \ar[r]^{}  &   0}
$$

Assume now $R=S/(f)$, where $S$ is a Gorenstein ring and $f$ is a non-zerodivisor of $S$. Let $S^{(m)}\twoheadrightarrow M_{1}$ be the composition of the canonical map $S^{(m)}\twoheadrightarrow R^{(m)}$ and the map $R^{(m)}\twoheadrightarrow M_{1}$ in (PF). Then a \emph{quasi-lifting} of $M$ with respect to the presentation $R=S/(f)$ is the $S$-module $E$ in the following short exact sequence:
$$\;\;\textnormal{(QL)}\;\;\; 0 \rightarrow E \rightarrow S^{(m)} \rightarrow M_{1} \rightarrow 0$$
Therefore the quasi-lifting of $M$ is unique, up to an isomorphism of $S$-modules.

We collect several properties of the pushforward and quasi-lifting from \cite{HJW}.

\begin{prop} (\cite[1.6 - 1.8]{HJW}) Let $R=S/(f)$ where $S$ is a Gorenstein ring and $f$ is a non-zerodivisor of $S$. Assume $M$ and $N$ are finitely-generated torsion-free $R$-modules. Let $M_{1}$ and $N_{1}$ denote the pushforwards and $E$ and $F$ the quasi-liftings of $M$ and $N$, respectively. Then one has the following properties:
\begin{enumerate}
\item Suppose $q\in \textnormal{Spec}(R)$ and $M_{q}$ is maximal Cohen-Macaulay over $R_{q}$. If $(M_{1})_{q}\neq 0$, then $(M_{1})_{q}$ is maximal Cohen-Macaulay over $R_{q}$.
\item Suppose $n$ is a positive integer. If $M$ satisfies $(S_{n})$ as an $R$-module, then $M_{1}$ satisfies $(S_{n-1})$ as an $R$-module.
\item There is a short exact sequence of $R$-modules: $0 \rightarrow M_{1} \rightarrow E/fE \rightarrow M \rightarrow 0$.
\item If $p\in \textnormal{Spec}(S)$ and $f\notin p$, then $E_{p}$ is free over $S_{p}$.
\item Suppose $p\in \textnormal{Spec}(S)$, $f\in p$ and $q=p/(f)$. If $M_{q}$ is free over $R_{q}$, then $E_{p}$ is free over $S_{p}$.
\item Suppose $p\in \textnormal{Spec}(S)$, $f\in p$ and $q=p/(f)$. If $(M_{1})_{q}\neq 0$, then $\textnormal{depth}_{S_{p}}(E_{p})= \textnormal{depth}_{R_{q}}((M_{1})_{q})+1$.
\item Suppose $S$ is a complete intersection ring and $v$ is a positive integer. Assume that both $M$ and $N$ satisfy $(S_{v})$ as $R$-modules and that $M\otimes_{R}N$ satisfies $(S_{v+1})$ as an $R$-module. If $\textnormal{Tor}^{R}_{i}(M,N)_{q}=0$ for all $i\geq 1$ and all $q\in X^{v}(R)$, then $E\otimes_{S}F$ satisfies $(S_{v})$.
\end{enumerate}
\end{prop}

The following proposition is embedded in the proofs of \cite[1.8]{HJW} and \cite[2.4]{HJW}. Here we include its proof for completeness.

\begin{prop}(\cite{HJW}) Let $R=S/(f)$ where $(S,\mathfrak{n})$ is a complete intersection and $f$ is a non-zerodivisor of $S$ contained in $\mathfrak{n}$. Assume $M$ and $N$ are finitely-generated torsion-free $R$-modules. Let $M_{1}$ and $N_{1}$ denote the pushforwards and $E$ and $F$ the quasi-liftings of $M$ and $N$, respectively.
\begin{enumerate}
\item $\textnormal{Tor}^{R}_{i}(E/fE,N)\cong \textnormal{Tor}^{S}_{i}(E,F)$ for all $i\geq 1$.
\item For each $i\in \ZZ$ there exists an exact sequence \\ $\textnormal{Tor}^{R}_{i+2}(E/fE,N) \rightarrow \textnormal{Tor}^{R}_{i+2}(M,N) \rightarrow  \textnormal{Tor}^{R}_{i+1}(M_{1},N) \rightarrow \textnormal{Tor}^{R}_{i+1}(E/fE,N) \\ \rightarrow \textnormal{Tor}^{R}_{i+1}(M,N)$.
\item Assume $\textnormal{Tor}^{R}_{i}(M,N)_{q}=0$ for all $i\geq 1$ and all $q \in X^{1}(R)$.
\begin{enumerate}
    \item If $M\otimes_{R}N$ is torsion-free, then $\textnormal{Tor}^{R}_{1}(M_{1},N)=0$.
    \item Assume $M\otimes_{R}N$ is reflexive. Then $M_{1}\otimes_{R}N$ is torsion-free. Moreover, if $\textnormal{Tor}^{S}_{i}(E,F)=0$ for all $i\geq 1$, then $\textnormal{Tor}^{R}_{i}(M,N)=0$ for all $i\geq 1$.
\end{enumerate}
    \item Let $w$ be a positive integer. Assume $M\otimes_{R}N$ is torsion-free, and that $\textnormal{Tor}^{R}_{i}(M,N)_{q}=0$ for all $i\geq 1$ and all $q\in X^{w}(R)$. Then $\textnormal{Tor}^{S}_{i}(E,F)_{p}=0$ for all $i\geq 1$ and all $p\in X^{w+1}(S)$.
    \end{enumerate}
\end{prop}

\begin{proof}
Consider the pushforward and quasi-lifting of $N$:
$$(3.2.1)\;\; 0 \rightarrow N \rightarrow R^{(n)} \rightarrow N_{1} \rightarrow 0$$
$$(3.2.2)\;\; 0 \rightarrow F \rightarrow S^{(n)} \rightarrow N_{1} \rightarrow 0$$
Tensoring $(3.2.1)$ with $E/fE$, we have that $\textnormal{Tor}^{R}_{i+1}(E/fE,N_{1})\cong \textnormal{Tor}^{R}_{i}(E/fE,N)$ for all $i\geq 1$. Therefore \cite[18, \textnormal{Lemma 2}]{Mat} and $(3.2.2)$ yield the isomorphism in $(1)$.

Statement $(2)$ follows at once from the exact sequence in Proposition $3.1(3)$.

For $(3)$, consider the pushforward of $M$:
$$(3.2.3)\;\;0 \rightarrow M \rightarrow R^{(m)} \rightarrow M_{1} \rightarrow 0$$
Tensoring $(3.2.3)$ with $N$, we get
$$(3.2.4)\;\;\textnormal{Tor}^{R}_{i}(M,N) \cong  \textnormal{Tor}^{R}_{i+1}(M_{1},N) \textnormal{ for all } i\geq 1.$$
Let $q\in X^{1}(R)$. Then $N_{q}$ is maximal Cohen-Macaulay over $R_{q}$. Moreover, $(3.2.4)$ implies that $\textnormal{Tor}^{R}_{i}(M_{1},N)_{q}=0$ for all $i\geq 2$. Therefore, by \cite[2.2]{Jo1}, we have
$$(3.2.5)\;\;\textnormal{Tor}^{R}_{i}(M_{1},N)_{q}=0\textnormal{ for all } i\geq 1.$$
Note that $(3.2.3)$ implies that there is an injection $\textnormal{Tor}^{R}_{1}(M_{1},N)\hookrightarrow M\otimes_{R}N$. Thus part $(a)$ follows from $(3.2.5)$. Assume now $M\otimes_{R}N$ is reflexive. We will prove that $M_{1}\otimes_{R}N$ is torsion-free.  Note that, if $\textnormal{dim}(R)=1$, then $\textnormal{Tor}^{R}_{i}(M_{1},N)=0$ for all $i\geq 1$. Therefore the claim follows from Proposition $3.1(1)$ and Theorem $2.7$. Thus we may assume $\textnormal{dim}(R)\geq 2$. Let $q$ be a prime ideal of $R$ such that $(M_{1}\otimes_{R}N)_{q}\neq 0$. Assume $\textnormal{dim}(R_{q})\leq 1$. Then, by $(3.2.5)$ and Theorem $2.7$, $(M_{1}\otimes_{R}N)_{q}$ is maximal Cohen-Macaulay. Assume now $\textnormal{dim}(R_{q})\geq 2$. Note that $(3.2.3)$ yields the following exact sequence:
$$(3.2.6)\;\;0\rightarrow M\otimes_{R}N \rightarrow N^{(m)} \rightarrow M_{1}\otimes_{R}N \rightarrow 0$$
Since $M\otimes_{R}N$ is reflexive, localizing $(3.2.6)$ at $q$, we see that the depth lemma implies $\textnormal{depth}_{R_{q}}((M_{1}\otimes_{R}N)_{q})\geq 1$. This proves that $M_{1}\otimes_{R}N$ is torsion-free. Suppose now $\textnormal{Tor}^{S}_{i}(E,F)=0$ for all $i\geq 1$. Then, by $(1)$ and $(2)$, we have
$$(3.2.7)\;\;\textnormal{Tor}^{R}_{i+2}(M,N) \cong  \textnormal{Tor}^{R}_{i+1}(M_{1},N) \textnormal{ for all } i\geq 0.$$
In particular, $\textnormal{Tor}^{R}_{2}(M,N)=0$. Note that, by $(3.2.4)$ and $(3.2.7)$, we have that $\textnormal{Tor}^{R}_{i}(M,N)\cong \textnormal{Tor}^{R}_{i+2}(M,N)$ for all $i\geq 1$. Since $\textnormal{Tor}^{R}_{1}(E/fE,N)=0$ by $(1)$, letting $i=-1$ in $(2)$, we see that $\textnormal{Tor}^{R}_{1}(M,N)=0$. Therefore $\textnormal{Tor}^{R}_{i}(M,N)=0$ for all $i\geq 1$. This proves $(3)$.

For $(4)$, let $p\in X^{w+1}(S)$. By Proposition $3.1(4)$, we may assume $f\in p$. Let $q=p/(f)$. Then $q\in X^{w}(R)$. Recall that, by $(3a)$, $\textnormal{Tor}^{R}_{1}(M_{1},N)=0$. Moreover, $(3.2.4)$ implies that $\textnormal{Tor}^{R}_{i}(M_{1},N)_{q}=0$ for all $i\geq 2$. Therefore $\textnormal{Tor}^{R}_{i}(M_{1},N)_{q}=0$ for all $i\geq 1$. Now the short exact sequence in Proposition $3.1(3)$ yields that $\textnormal{Tor}^{R}_{i}(E/fE,N)_{q}=0$ for all $i\geq 1$. Thus $(4)$ follows from the isomorphism in $(1)$.
\end{proof}

Our results are motivated by the following theorem due to H.\ Dao.

\begin{theorem}(\cite[7.6]{Da2}) Let $(R,\mathfrak m)$ be an admissible local complete intersection of codimension $c$, and let $M$ and $N$ be finitely generated $R$-modules. Assume $M$ is free on $X^{c}(R)$, $M$ and $N$ satisfy $(S_{c})$ and $M\otimes_{R}N$ satisfies $(S_{c+1})$. Then $\textnormal{Tor}^{R}_{i}(M,N)=0$ for all $i\geq 1$.
\end{theorem}

Although Theorem $3.3$ is a powerful tool, it has no content when $c\geq\textnormal{dim}(R)$. (The assumption that $M$ is free on $X^{c}(R)$ forces $M$ to be free). We will prove variations of this result that give useful information even when $c\geq\textnormal{dim}(R)$.
\vspace{0.05in}

Note that, if one assumes $M$ is free on $X^{c-1}(R)$ instead of $X^{c}(R)$ in
Theorem $3.3$, then it is not necessarily true that
$\textnormal{Tor}^{R}_{i}(M,N)=0$ for all $i\geq 1$: Let $k$ be a
field, $R=k[[X,Y]]/(XY)$ and $M=R/(x)$. Then $M$ is a maximal
Cohen-Macaulay vector bundle and $\textnormal{Tor}^{R}_{i}(M,M)\neq
0$ if and only if $i$ is a positive odd integer, or zero. Assuming $M$ is free on $X^{c-1}(R)$, we show in Theorem $3.4$ that non-vanishing homology can occur only if the modules considered have maximal complexities. This improves Theorem $3.3$ for modules of small
complexities (see also Corollary $3.16$).

\begin{theorem} Let $(R,\mathfrak m)$ be a local ring such that $\hat{R}=S/(\underline{f})$ where $(S,\mathfrak{n})$ is a complete unramified regular local ring and $\underline{f}=f_{1},f_{2},\dots,f_{c}$ is a regular sequence of $S$ contained in $\mathfrak{n}^{2}$. Let $M$ and $N$ be finitely generated $R$-modules. Assume the following conditions hold:
\begin{enumerate}
\item $M$ and $N$ satisfy $(S_{c-1})$.
\item $M\otimes_{R}N$ satisfies $(S_{c})$.
\item If $c\geq 2$, then assume further that $\textnormal{Tor}^{R}_{i}(M,N)_{q}=0$ for all $i\geq 1$ and all $q\in X^{c-1}(R)$ (e.g., $M$ is free on $X^{c-1}(R)$).
\end{enumerate}
Then either $\textnormal{cx}_{R}(M)=\textnormal{cx}_{R}(N)=c$, or $\textnormal{Tor}^{R}_{i}(M,N)=0$ for all $i\geq 1$.
\end{theorem}

\begin{proof} Note that if $R\rightarrow A$ is a flat local homomorphism of Gorenstein rings, $b$ is a positive integer, and $X$ is a finitely generated $R$-module satisfying $(S_{b})$ as an $R$-module, then $X\otimes_{R}A$ satisfies $(S_{b})$ as an $A$-module. (This follows from Proposition $3.1(2)$; see \cite[1.3]{LW} or the proof of \cite[3.8]{EG} for a stronger result.) Moreover, an unramified regular local ring $(S,n)$ remains unramified when we extend its residue field by using the faithfully flat extension $S \hookrightarrow S[z]_{\mathfrak n S[z]}$ where $z$ is an indeterminate over $S$. Therefore, without loss of generality, we may assume $R$ is complete and has infinite residue field. We will use the same notations for the pushforwards and quasi-liftings of $M$ and $N$ as in the proof of Proposition $3.2$.

If $c=0$, then $\textnormal{cx}_{R}(M)=\textnormal{cx}_{R}(N)=0$, and so we may assume $c\geq 1$. Without loss of generality, we will assume $\textnormal{cx}_{R}(M)<c$ and prove that $\textnormal{Tor}^{R}_{i}(M,N)=0$ for all $i\geq 1$. We proceed by induction on $c$. Suppose $c=1$. Then, by assumption, $M$ has finite projective dimension. Since $M\otimes_{R}N$ is torsion-free, \cite[2.3]{HW2} implies that $\textnormal{Tor}^{R}_{i}(M,N)=0$ for all $i\geq 1$. Assume now $c\geq 2$. By the proof of \cite[1.3]{Jo1}, there exists a regular sequence $\underline{x}=x_{1},x_{2}, \dots, x_{c}$ generating $(\underline{f})$ such that $R=R_{1}/(x)$ and $\textnormal{cx}_{R_{1}}(M_{1})<\textnormal{codim}(R_{1})=c-1$, where
$R_{1}=S/(x_{2},x_{3},\dots, x_{c})$ and $x=x_{1}$. It follows that $\textnormal{cx}_{R_{1}}(E)<\textnormal{codim}(R_{1})$. Note that $(2)$ and $(6)$ of Proposition $3.1$ imply that $E$ and $F$ satisfy $(S_{c-1})$. Moreover, letting $w=c-1$ in Proposition $3.2(4)$, we have $\textnormal{Tor}^{R_{1}}_{i}(E,F)_{p}=0$ for all $i\geq 1$ and all $p\in X^{c}(R_{1})$. Finally, setting $v=c-1$ in Proposition $3.1(7)$, we conclude that $E\otimes_{R_{1}}F$ satisfies $(S_{c-1})$. Hence, if we replace $M$ and $N$ by $E$ and $F$ and $c$ by $c-1$, the induction hypothesis implies that $\textnormal{Tor}^{R_{1}}_{i}(E,F)=0$ for all $i\geq 1$. Therefore, by Proposition $3.2(3b)$, $\textnormal{Tor}^{R}_{i}(M,N)=0$ for all $i\geq 1$.
\end{proof}

\begin{corollary} Let $(R,\mathfrak m)$ be a local ring such that $\hat{R}=S/(\underline{f})$ where $(S,\mathfrak{n})$ is a complete unramified regular local ring and $\underline{f}=f_{1},f_{2},\dots,f_{c}$ is a regular sequence of $S$ contained in $\mathfrak{n}^{2}$. Let $M$ be a finitely generated $R$-module. Assume the following conditions hold:
\begin{enumerate}
  \item $M$ satisfies $(S_{c-1})$.
  \item $M$ is free on $X^{c-1}(R)$.
  \item $M\otimes_{R}M$ satisfies $(S_{c})$.
\end{enumerate}
Then either $\textnormal{cx}_{R}(M)=c$, or $M$ has finite projective dimension.
\end{corollary}

\begin{proof} The case where $c\leq 1$ is trivial. Suppose $c\geq 2$ and $\textnormal{cx}_{R}(M)<c$. Then, by Theorem $3.4$, $\textnormal{Tor}^{R}_{i}(M,M)=0$ for all $i\geq 1$. Therefore, by \cite[1.2]{Jo2}, $M$ has finite projective dimension. This proves the corollary.
\end{proof}

\begin{remark} Note that, in Corollary $3.5$, if $c\geq 1$ and $\textnormal{cx}_{R}(M)<c$, then $\textnormal{Tor}^{R}_{i}(M,M)=0$ for all $i\geq 1$ and hence the localization of the depth formula of Theorem $2.7$ shows that $M$ satisfies $(S_{c})$. It is not known (at least to the author) whether one can conclude the same thing for the module $M$ in Theorem $3.4$. More specifically, if $(R,\mathfrak m)$ is a local complete intersection, and $M$ and $N$ are non-zero finitely generated $R$-modules such that $M\otimes_{R}N$ satisfies $(S_{n})$ for some $n$ and $\textnormal{Tor}^{R}_{i}(M,N)=0$ for all $i\geq 1$, then does $M$ satisfy $(S_{n})$?
\cite[2.8]{ArY} asserts a positive answer to this question, but the proof is flawed. The localization of the depth formula at a prime ideal which is not in the support of $N$ does not reveal anything about the depth of $M$.
\end{remark}

Next we examine Theorem $3.4$ when one of the modules considered is maximal Cohen-Macaulay. We will use the following variation of Theorem $2.4$.

\begin{prop} \footnote{After the submission of this paper, a more general version of Proposition $3.7$ appeared in \cite{BJ}.} Let $(R,m)$ be a $d$-dimensional local complete intersection ring, and let $M$ and $N$ be finitely generated $R$-modules.
Set $r=\textnormal{min}\left\{\textnormal{cx}_{R}(M),\textnormal{cx}_{R}(N)\right\}$ and $b=\textnormal{max}\left\{\textnormal{depth}_{R}(M), \textnormal{depth}_{R}(N)\right\}$. Assume $r\geq 1$ and $$\textnormal{Tor}^{R}_{n}(M,N)=\textnormal{Tor}^{R}_{n+1}(M,N)= \ldots =\textnormal{Tor}^{R}_{n+r-1}(M,N)=0$$ for some $n\geq d-b+1$.
\begin{enumerate}
    \item If $r$ is odd, then $\textnormal{Tor}^{R}_{n+2i}(M,N)=0$ for all $i\geq 0$.
    \item If $r$ is even, then $\textnormal{Tor}^{R}_{n+2i+1}(M,N)=0$ for all $i\geq 0$.
\end{enumerate}
\end{prop}

\begin{proof} Without loss of generality we may assume $r=\textnormal{cx}_{R}(M)$. Moreover, by passing to $R[z]_{\mathfrak m R[z]}$ and then completing, we may assume that $R$ is complete and has infinite residue field. We proceed by induction on $r$. Assume $r=1$. Then, by Theorem $2.6$, $\textnormal{Tor}^{R}_{i}(M,N)\cong \textnormal{Tor}^{R}_{i+2}(M,N)$ for all $i\geq d-b+1$. Since $\textnormal{Tor}^{R}_{n}(M,N)=0$ by assumption, we conclude that $\textnormal{Tor}^{R}_{n+2i}(M,N)=0$ for all $i\geq 0$. Assume now that $r\geq 2$. Since $R$ is complete, the proof of \cite[2.1.i]{Be} (cf. also the proofs of \cite[7.8 \textnormal{and} 8.6(2)]{AGP}) provides a short exact sequence
$$(3.7.1)\;\;\; 0 \rightarrow M \rightarrow K \rightarrow syz^{R}_{1}(M) \rightarrow 0$$ where $K$ is a finitely generated $R$-module such that $\textnormal{cx}_{R}(K)=r-1$ and $\textnormal{depth}_{R}(K)\\=\textnormal{depth}_{R}(M)$.
We now have the following exact sequence induced by $(3.7.1)$:
\vspace{0.1in}

$(3.7.2)\;\;\; \textnormal{Tor}^{R}_{j+1}(K,N) \rightarrow \textnormal{Tor}^{R}_{j+2}(M,N) \rightarrow \textnormal{Tor}^{R}_{j}(M,N) \rightarrow \textnormal{Tor}^{R}_{j}(K,N)\rightarrow \\ \textnormal{Tor}^{R}_{j+1}(M,N)$ \; for all $j\geq 1$.
\vspace{0.1in}

\noindent This shows that $\textnormal{Tor}^{R}_{n}(K,N)=\textnormal{Tor}^{R}_{n+1}(K,N)= \ldots =\textnormal{Tor}^{R}_{n+r-2}(K,N)=0$.
If $r$ is even, then the induction hypothesis implies $\textnormal{Tor}^{R}_{n+2i}(K,N)=0$ for all $i\geq 0$. Therefore, using $(3.7.2)$, we have an injection $\textnormal{Tor}^{R}_{n+2i+1}(M,N) \hookrightarrow \textnormal{Tor}^{R}_{n+2i-1}(M,N)$ for all $i\geq 1$, and $(2)$ follows. Similarly, if $r$ is odd, then the induction hypothesis implies that $\textnormal{Tor}^{R}_{n+2i+1}(K,N)=0$ for all $i\geq 0$. Hence, by $(3.7.2)$, $\textnormal{Tor}^{R}_{n+2i+2}(M,N) \hookrightarrow \textnormal{Tor}^{R}_{n+2i}(M,N)$ is an injection for all $i\geq 0$, and we have $(1)$.
\end{proof}

In the proof of Theorem $3.9$, we will use the following result: If $(R,\mathfrak m)$ is a local complete intersection, $M$ a maximal Cohen-Macaulay $R$-module and $N$ is a finitely generated $R$-module that has finite projective dimension, then $\textnormal{Tor}^{R}_{i}(M,N)=0$ for all $i\geq 1$. Note that this follows from Theorem $2.4$, or the fact that, over a Gorenstein ring $R$, a maximal Cohen-Macaulay $R$-module is a $d$th syzygy where $d=\textnormal{dim}(R)$. It is worth noting that this result also holds over any local ring \cite[2.2]{Yo}. Here we include an elementary proof for the general case and refer the interested reader to \cite[4.9]{AB} for a more general result.

\begin{theorem} (\cite[2.2]{Yo}) Let $(R,m)$ be a local ring, and let $M$ and $N$ be non-zero finitely generated $R$-modules. If $M$ is maximal Cohen-Macaulay and $N$ has finite projective dimension, then $\textnormal{Tor}^{R}_{i}(M,N)=0$ for $i\geq 1$.
\end{theorem}

\begin{proof} We will first show that $\textnormal{Tor}^{R}_{1}(M,N)=0$ by induction on $\textnormal{dim}(R)$.
Note that, by the Auslander-Buchsbaum equality, the result holds if $\textnormal{depth}(R)=0$. In particular the case where $\textnormal{dim}(R)=0$ holds. Suppose now $\textnormal{depth}(R)>0$. Then, by the induction hypothesis, $\textnormal{Tor}^{R}_{1}(M,N)$ has finite length. Let $N'=\textnormal{syz}^{R}_{1}(N)$ and choose a non-zerodivisor on $R$ and $N'$. Since $M/xM$ is maximal Cohen-Macaulay and $N'/xN'$ has finite projective dimension over $R/xR$, the induction hypothesis implies that $\textnormal{Tor}^{R/xR}_{1}(M/xM,N'/xN')=0$. Consider the short exact sequence
$$(3.8.1)\;\; 0 \rightarrow M \stackrel{x}{\rightarrow} M \rightarrow M/xM \rightarrow 0 $$
Tensor $(3.8.1)$ with $N'$ to get the exact sequence
$$(3.8.2)\;\; \textnormal{Tor}^{R}_{1}(M/xM,N') \rightarrow M\otimes_{R}N' \stackrel{x}{\rightarrow}  M\otimes_{R}N'$$
Note that $\textnormal{Tor}^{R/xR}_{1}(M/xM,N'/xN') \cong \textnormal{Tor}^{R}_{1}(M/xM,N')$ by \cite[18,\textnormal{ Lemma} 2]{Mat}. Thus $\textnormal{Tor}^{R}_{1}(M/xM,N')=0$. Hence $(3.8.2)$ shows that $\textnormal{depth}(M\otimes_{R}N')>0$. Now consider the short exact sequence
$$(3.8.3)\;\; 0 \rightarrow N' \rightarrow R^{(t)} \rightarrow N \rightarrow 0$$
Since $\textnormal{depth}(M\otimes_{R}N')>0$ and $\textnormal{Tor}^{R}_{1}(M,N)$ has finite length over $R$, tensoring $(3.8.3)$ with $M$, we conclude that $\textnormal{Tor}^{R}_{1}(M,N)=0$. Now induction on the projective dimension of $N$ shows that $\textnormal{Tor}^{R}_{i}(M,N)=0$ for $i\geq 1$.
\end{proof}

\begin{theorem} Let $(R,\mathfrak m)$ be a local complete intersection, and let $M$ and $N$ be finitely generated $R$-modules. Assume $M$ is maximal Cohen-Macaulay. Set $r=\textnormal{min}\left\{\textnormal{cx}_{R}(M),\textnormal{cx}_{R}(N) \right\}$.
\begin{enumerate}
\item Assume $M$ is free on $X^{r}(R)$, $N$ satisfies $(S_{r})$ and $M\otimes_{R}N$ satisfies $(S_{r+1})$. Then
$\textnormal{Tor}^{R}_{i}(M,N)=0$ for all $i\geq 1$.
\item Assume $M$ is free on $X^{r-1}(R)$, $N$ satisfies $(S_{r-1})$ and $M\otimes_{R}N$ satisfies $(S_{r})$. Then
$\textnormal{Tor}^{R}_{i}(M,N)=0$ for all even $i\geq 2$. Furthermore, if $\textnormal{Tor}^{R}_{j}(M,N)=0$ for some odd $j\geq 1$, then $\textnormal{Tor}^{R}_{i}(M,N)=0$ for all $i\geq 1$.
\end{enumerate}
\end{theorem}

\begin{proof} As in the proof of Theorem $3.4$, we may assume $R$ is complete and has infinite residue field. If $M$ has finite projective dimension, then $M$ is free by the Auslander-Buchsbaum equality so there is nothing to prove. If $N$ has finite projective dimension, then $\textnormal{Tor}^{R}_{i}(M,N)=0$ for all $i\geq 1$ by Theorem $3.8$. Thus we may assume $\textnormal{cx}_{R}(M)>0$ and $\textnormal{cx}_{R}(N)>0$. We will use the same notations for the pushforwards and quasi-liftings of $M$ and $N$ as in the proof of Proposition $3.2$.

We set $M_{0}=M$ and consider the pushforwards for $i=1,2,\ldots, r+1$:
$$(3.9.1)\;\; 0 \rightarrow M_{i-1} \rightarrow G_{i} \rightarrow M_{i} \rightarrow 0 $$
Note that, since we assume $M$ is maximal Cohen-Macaulay, Proposition $3.1(1)$ implies that $M_{i}$ is maximal Cohen-Macaulay for all $i=1,2,\ldots, r+1$.
\vspace{0.1in}

\noindent $(1)$ The assumptions and \cite[2.1]{HJW} imply that $\textnormal{Tor}^{R}_{i}(M_{r+1},N)=0$ for all $i=1,2,\ldots,r+1$. Since $\textnormal{min}\left\{\textnormal{cx}_{R}(M_{r+1}),\textnormal{cx}_{R}(N)\right\}=r$, $\textnormal{Tor}^{R}_{i}(M_{r+1},N)=0$ for all $i\geq 1$ by Theorem $2.4$. This implies that $\textnormal{Tor}^{R}_{i}(M,N)=0$ for all $i\geq 1$.
\vspace{0.1in}

\noindent $(2)$ The assumptions and \cite[2.1]{HJW} imply that $\textnormal{Tor}^{R}_{i}(M_{r},N)=0$ for all $i=1,2,\ldots,r$. Therefore, by Proposition $3.7$, $\textnormal{Tor}^{R}_{i}(M_{r},N)=0$ for all even $i\geq 2$ if $r$ is even, and $\textnormal{Tor}^{R}_{i}(M_{r},N)=0$ for all odd $i\geq 1$ if $r$ is odd. Hence, by shifting along the sequences $(3.9.1)$, we conclude that $\textnormal{Tor}^{R}_{i}(M,N)=0$ for all even $i\geq 2$.

Suppose now $\textnormal{Tor}^{R}_{j}(M,N)=0$ for some odd $j\geq 1$. To prove the second claim in $(2)$, we proceed by induction on $r$. Assume $r=1$. Then, by Theorem $2.6$, $\textnormal{Tor}^{R}_{i}(M,N) \cong \textnormal{Tor}^{R}_{i+2}(M,N)$ for all $i\geq 1$, and hence the result follows. Assume now $r\geq 2$. Recall that $M_{1}$ and $N_{1}$ denote the pushforwards of $M$ and $N$, respectively. (Note that, since $r\geq 2$, we can construct the pushforward of $N$.) As in the proof of Theorem $3.4$, we choose, using Theorem $2.5$, a complete intersection $S$, and a non-zerodivisor $f$ of $S$ such that $R=S/(f)$ and $\textnormal{min}\left\{\textnormal{cx}_{S}(M_{1}),\textnormal{cx}_{S}(N_{1}) \right\}=r-1$. Now, with respect to the presentation $R=S/(f)$, we construct the quasi-liftings $E$ and $F$ of $M$ and $N$, respectively:
$$(3.9.2)\;\; 0 \rightarrow E \rightarrow S^{(m)} \rightarrow M_{1} \rightarrow 0$$
$$(3.9.3)\;\; 0 \rightarrow F \rightarrow S^{(n)} \rightarrow N_{1} \rightarrow 0$$
Thus $\textnormal{min}\left\{\textnormal{cx}_{S}(E),\textnormal{cx}_{S}(F) \right\}=r-1$. Note that  $\textnormal{Tor}^{R}_{i}(M_{1},N)=0$ for all odd $i\geq 1$. Therefore, by $(1)$ and $(2)$ of Proposition $3.2$, we see that $\textnormal{Tor}^{S}_{j}(E,F)=0$. Now, replacing $M$ and $N$ by $E$ and $F$, and using the induction hypothesis with Proposition $3.2(3b)$, we conclude that $\textnormal{Tor}^{R}_{i}(M,N)=0$ for all $i\geq 1$.
\end{proof}

\begin{remark} It is shown in \cite[4.1]{HW1} that the assumptions, in $(1)$ and $(2)$ of Theorem $3.9$, that $M\otimes_{R}N$ satisfies $(S_{r})$, respectively $(S_{r+1})$, cannot be removed.
\end{remark}

It should be pointed out that if $M$ and $N$ are two finitely generated modules over a complete intersection $R$ such that $M$ is maximal Cohen-Macaulay and $\textnormal{Tor}^{R}_{i}(M,N)=0$ for all even $i\geq 2$, then the vanishing of $\textnormal{Tor}^{R}_{j}(M,N)$ for some odd $j\geq 1$ does not in general imply $\textnormal{Tor}^{R}_{i}(M,N)=0$ for all $i\geq 1$. The following example, verified by \emph{Macaulay 2} \cite{GS}, is a special case of \cite[4.1]{Jo1}. (See also Corollary $4.7$.)

\begin{example} (\cite[4.1]{Jo1}) Let $k$ be a field and put $R=k[[X,Y,Z,U]]/(XY,ZU)$. Then $R$ is a complete intersection of dimension two and codimension two. Let $M=R/(y,u)$, and let $N$ be the cokernel of the following map:
$$\xymatrixcolsep{16pt}\xymatrix{R^{(2)}\ar[rrr]^-{\left[\begin{array}{cc} 0 &  u \\
-z & x \\ y & 0\end{array}\right]} & &
& R^{(3)}} $$
Then $M$ and $N$ are maximal Cohen-Macaulay, $\textnormal{Tor}^{R}_{1}(M,N)=\textnormal{Tor}^{R}_{2}(M,N)=0$ and $\textnormal{cx}_{R}(M)=\textnormal{cx}_{R}(N)=2$. Moreover, by Proposition $3.7$, $\textnormal{Tor}^{R}_{i}(M,N)=0$ for all even $i\geq 2$. Therefore, if there is an odd $j\geq 3$ such that $\textnormal{Tor}^{R}_{j}(M,N)=0$, then Theorem $2.2$ implies $\textnormal{Tor}^{R}_{i}(M,N)=0$ for all $i \gg 0$. This shows, by \cite[2.1]{Mi}, that $\textnormal{cx}_{R}(M)+\textnormal{cx}_{R}(N)\leq 2$, which is false. Thus $\textnormal{Tor}^{R}_{i}(M,N)\neq 0$ for all odd $i\geq 3$.
\end{example}

As an immediate corollary of Theorem $3.9$, we have:

\begin{corollary} (Theorem $1.1$) Let $(R,\mathfrak m)$ be local complete intersection, and let $M$ and $N$ be finitely generated $R$-modules. Assume $M$, $N$ and $M\otimes_{R}N$ are maximal Cohen-Macaulay. Set $r=\textnormal{min}\left\{\textnormal{cx}_{R}(M),\textnormal{cx}_{R}(N) \right\}$.
\begin{enumerate}
\item If $M$ is free on $X^{r}(R)$, then $\textnormal{Tor}^{R}_{i}(M,N)=0$ for all $i\geq 1$.
\item Assume $M$ is free on $X^{r-1}(R)$. Then $\textnormal{Tor}^{R}_{i}(M,N)=0$ for all even $i\geq 2$. Moreover, if $\textnormal{Tor}^{R}_{j}(M,N)=0$ for some odd $j\geq 1$, then $\textnormal{Tor}^{R}_{i}(M,N)=0$ for all $i\geq 1$.
\end{enumerate}
\end{corollary}

The assumptions in $(1)$ and $(2)$ of Corollary $3.12$ that $M$ is free on $X^{r-1}(R)$, respectively on $X^{r}(R)$, cannot be removed.

\begin{example} Let $R$ and $M$ be as in Example $3.11$, and let $q=(y,u,x)$. Then $\textnormal{dim}(R_{q})=1$ and $M_{q}=R_{q}/(y)$ is not a free $R_{q}$-module. Thus $M$ is not a vector bundle. It can be checked that a minimal resolution of $M$ is:\\
$$\xymatrix{\ldots \ar[r] & R^{(4)} \ar[rrr]^-{\left[\begin{array}{cccc} u & -y & 0 & 0\\
0 & z & x & 0 \\ 0 &  0 & u & y \end{array}\right]}& &
& R^{(3)}   \ar[rrr]^-{\left[\begin{array}{ccc}0 & -u & x \\ z & y & 0\end{array}\right]}& &
& R^{(2)}   \ar[rr]^-{\left[\begin{array}{cc}y & u \end{array}\right]} &
& R \ar[r]
&0}$$
Using the resolution above, we see that $\textnormal{Tor}^{R}_{2}(M,M)\neq 0$.
\end{example}

\begin{example} Let $R$ be as in Example $3.11$, and let $M=R/(x)$ and $N=R/(xz)$. Then $M$, $N$ and $M\otimes_{R}N$ are maximal Cohen-Macaulay.
A minimal resolution of $M$ is:
$$\xymatrix{\ldots \ar[r]^y& R \ar[r]^x &  R \ar[r]^y &  R \ar[r]^x & R  \ar[r]
&0}$$
It is easy to see that $\textnormal{Tor}^{R}_{1}(M,N)\neq 0$, $\textnormal{Tor}^{R}_{2}(M,N)=0$ and $\textnormal{Tor}^{R}_{i}(M,N)\cong \textnormal{Tor}^{R}_{i+2}(M,N)$ for all $i\geq 1$. One can also see that $\textnormal{Tor}^{R}_{1}(M,N)\cong R/(x,y,u) \cong k[[Z]]$. In particular, $\textnormal{depth}(\textnormal{Tor}^{R}_{i}(M,N))=1$ if $i$ is a positive odd integer. Hence $M$ and $N$ are not vector bundles.

\end{example}

Our next theorem can be established by modifying the proof of \cite[7.6]{Da2} (stated as Theorem $3.3$). Here we give a different proof using the quasi-liftings as in Theorem $3.4$ and Theorem $3.9$. We will use it to make a further observation in Corollary $3.16$.

\begin{theorem} (H. Dao) Let $(R,\mathfrak m)$ be a local ring such that $\hat{R}=S/(\underline{f})$
where $(S,\mathfrak{n})$ is a complete unramified regular local ring and $\underline{f}=f_{1},f_{2},\dots,f_{c}$ is a regular sequence of $S$
contained in $\mathfrak{n}^{2}$. Let $M$ and $N$ be finitely generated $R$-modules, and $n$ be an integer such that $n\neq c$ if $n$ is positive. Assume the following conditions hold:
\begin{enumerate}
\item $M$ and $N$ satisfy $(S_{c-n})$.
\item $M$ is free on $X^{c-n}(R)$.
\item $M\otimes_{R}N$ satisfies $(S_{c-n+1})$.
\item $\textnormal{Tor}^{R}_{1}(M,N)=\textnormal{Tor}^{R}_{2}(M,N)= \ldots =\textnormal{Tor}^{R}_{n}(M,N)=0$.
\end{enumerate}
Then $\textnormal{Tor}^{R}_{i}(M,N)=0$ for all $i\geq 1$.
\end{theorem}

\begin{proof} Without loss of generality we may assume $R$ is complete.  We will use the same notations for the pushforwards and quasi-liftings of $M$ and $N$ as in the proof of Proposition $3.2$. Note that, if $n\leq 0$, then the result follows from \cite[7.6]{Da2}. Moreover, if $c<n$, then $(4)$ and Theorem $2.2$ imply that $\textnormal{Tor}^{R}_{i}(M,N)=0$ for all $i\geq 1$. Therefore we may assume $c>n\geq 1$.

Assume $c=n+1$. Then $M$ and $N$ are torsion-free, $M$ is free on $X^{1}(R)$, $M\otimes_{R}N$ is reflexive and $\textnormal{Tor}^{R}_{1}(M,N)=\textnormal{Tor}^{R}_{2}(M,N)= \ldots =\textnormal{Tor}^{R}_{c-1}(M,N)=0$. Consider the pushforward of $M$:
$$(3.15.1)\;\; 0\rightarrow M \rightarrow R^{(m)} \rightarrow M_{1} \rightarrow 0$$
Note that, by Proposition $3.2(3)$, $\textnormal{Tor}^{R}_{1}(M_{1},N)=0$ and $M_{1}\otimes_{R}N$ is torsion-free. Moreover, since $c\geq 2$, we have
$$(3.15.2)\;\; \textnormal{Tor}^{R}_{1}(M_{1},N)=\textnormal{Tor}^{R}_{2}(M_{1},N)=\ldots =\textnormal{Tor}^{R}_{c}(M_{1},N)=0.$$
We proceed by induction on $d=\text{dim}(R)$. The case where $d\leq 1$ follows from the fact that $M$ is free on $X^{1}(R)$. Assume $d\geq 2$. Then the induction hypothesis and $(3.15.1)$ imply that $\textnormal{Tor}^{R}_{i}(M_{1},N)$ has finite length for all $i\geq 1$. (If $R_{p}$ has codimension less than $c$, we use Theorem $2.2$ and $(3.15.2)$) Now applying Theorem $2.8$ to $M_{1}$ and $N$, we conclude that $\textnormal{Tor}^{R}_{i}(M_{1},N)=0$ for all $i\geq 1$. This proves the case where $c=n+1$.

Assume now $c\geq n+2$. Let $R=S/(f)$ where $S$ is an unramified complete intersection of codimension $c-1$, and $f$ is a non-zerodivisor of $S$.
Then $E$ and $F$ satisfy $(S_{c-n})$ and $E$ is free on $X^{c-n+1}(S)$ (cf. Proposition $3.1$). Moreover, by Proposition $3.1(7)$, $E\otimes_{S}F$ satisfies $(S_{c-n})$. Note also that $(1)$ and $(2)$ of Proposition $3.2$ show that $\textnormal{Tor}^{S}_{1}(E,F)=\textnormal{Tor}^{S}_{2}(E,F)= \ldots =\textnormal{Tor}^{S}_{n}(E,F)=0$. Therefore, replacing $M$ and $N$ by $E$ and $F$ and $c$ by $c-1$, the induction hypothesis on $c$ implies that $\textnormal{Tor}^{S}_{i}(E,F)=0$ for all $i\geq 1$. Thus, by Proposition $3.2(3b)$, $\textnormal{Tor}^{R}_{i}(M,N)=0$ for all $i\geq 1$.
\end{proof}

As a corollary of Theorem $3.4$ and Theorem $3.15$ we have:

\begin{corollary} Let $(R,\mathfrak m)$ be a local ring such that $\hat{R}=S/(\underline{f})$ where $(S,\mathfrak{n})$ is a complete unramified regular local ring and $\underline{f}=f_{1},f_{2},\dots,f_{c}$, for $c\neq 1$, is a regular sequence of $S$ contained in $\mathfrak{n}^{2}$. Let $M$ and $N$ be finitely generated $R$-modules. Assume the following conditions hold:
\begin{enumerate}
\item $M$ and $N$ satisfy $(S_{c-1})$.
\item $M$ is free on $X^{c-1}(R)$.
\item $M\otimes_{R}N$ satisfies $(S_{c})$.
\end{enumerate}
Then either $(a)$ $\textnormal{cx}_{R}(M)=\textnormal{cx}_{R}(N)=c$ and $\textnormal{Tor}^{R}_{1}(M,N)\neq 0$, or $(b)$ $\textnormal{Tor}^{R}_{i}(M,N)=0$ for all $i\geq 1$.
\end{corollary}

We do not know whether Theorem $3.15$ holds if $c=n\geq 1$. In particular, it seems reasonable to ask the following question (see also \cite[4.1]{Da1}.):

\begin{question} Let $(R,\mathfrak m)$ be a local ring such that $\hat{R}=S/(f)$ where $(S,\mathfrak{n})$ is a complete unramified regular local ring and $0 \neq f \in \mathfrak{n}^{2}$. Let $M$ and $N$ be finitely generated $R$-modules such that $M$ is free on $X^{0}(R)$ and $M\otimes_{R}N$ is torsion-free. If $\textnormal{Tor}^{R}_{1}(M,N)=0$, then is $\textnormal{Tor}^{R}_{i}(M,N)=0$ for all $i\geq 1$ ?
\end{question}

\section{Some Further Applications }
In this section we present some of the consequences of Theorem $2.8$ and the main theorems in \cite{HJW}. These results give useful information for the vanishing of Tor over local complete intersections when the modules considered have maximal complexities (cf. \cite[6.8]{Da2}).

We will first prove in Proposition $4.9$ that, over a local complete
intersection (quotient of an unramified regular local ring) of
codimension $c\geq 2$, vanishing of the first $c$ consecutive
$\textnormal{Tor}^{R}_{i}(M,N)$ implies the depth formula, provided
the modules $M$, $N$ and $M\otimes_{R}N$ are Cohen-Macaulay. The
motivation for such a result comes from Theorem $4.1$. Next we
make an observation about a question of Huneke and Wiegand (see
Question $4.16$ and Proposition $4.17$.) Finally we finish this
section with two more applications of the pushforward (see
Propositions $4.20$ and $4.22$.)

We start by recording the following theorem from an unpublished note of C. Huneke and R. Wiegand, used with their permission.

\begin{theorem} (C. Huneke - R. Wiegand) Let $(R,\mathfrak m)$ be a $d$-dimensional local ring such that $\hat{R}=S/(f)$ where $(S,\mathfrak{n})$ is a complete unramified regular local ring and $0 \neq f \in \mathfrak{n}^{2}$. Let $M$ and $N$ be finitely generated $R$-modules. Assume the following conditions hold:
\begin{enumerate}
\item $M$, $N$ and $M\otimes_{R}N$ are Cohen-Macaulay.
\item $\textnormal{dim}(M)+\textnormal{dim}(N)\leq d$.
\item $\textnormal{Tor}^{R}_{1}(M,N)=0$.
\end{enumerate}
Then  $\textnormal{Tor}^{R}_{i}(M,N)=0$ for all $i\geq 1$.
\end{theorem}

\begin{proof} Without loss of generality, we may assume $R$ is complete. We use induction on $n:=\text{depth}(M\otimes_{R}N)$. If $n=0$, then the result follows from Theorem $2.3$. So we assume $n>0$. Therefore $M$ and $N$ have positive depth. If $S$ is equicharacteristic, put $J=\mathfrak n^{2}$. Otherwise, let $p$ be the characteristic of $S/ \mathfrak n$, and let $J=\mathfrak n^{2}+Sp$. Since $d\geq 1$, $\mathfrak n\neq J$. Let $I=J/(f)$, and choose $x\in \mathfrak m-I$ such that $x$ is a non-zerodivisor on $M$, $N$ and $M\otimes_{R}N$. Let $\overline{X}=X/xX$ for an $R$-module $X$. Then $\overline{M}$, $\overline{N}$ and $\overline{M\otimes_{R}N}$ are Cohen-Macaulay over $\overline{R}$. Moreover, $\text{dim}(\overline{M})+\text{dim}(\overline{N})\leq d-2$. Lifting $x$ to $y\in S$, we have $y\in \mathfrak n- \mathfrak n^{2}$, so that $S/(y)$ is a regular local ring. To see that it is unramified, suppose $S$ is not equicharacteristic, so that $y\notin \mathfrak n^{2}+Sp$. If $S/(y)$ were ramified, we would have $p\in \mathfrak n^{2}+Sy$, say $p=b+sy$, with $b\in \mathfrak n^{2}$ and $s\in S$. Then $s\notin \mathfrak n$, since $p\notin \mathfrak n^{2}$. It follows that $y\in \mathfrak n^{2}+Sp$, contradiction. Now $R/(x)=S/(f,y)$, so the induction hypothesis applies. Note that, since $x$ is a non-zerodivisor on $M\otimes_{R}N$, the short exact sequence $$(4.1.1)\;\; 0\rightarrow M \stackrel{x}{\rightarrow} M \rightarrow \overline{M} \rightarrow 0$$ shows that $\textnormal{Tor}^{R}_{1}(\overline{M},N)=0$. Since $\textnormal{Tor}^{\overline{R}}_{i}(\overline{M},\overline{N})\cong \textnormal{Tor}^{R}_{i}(\overline{M},N)$ for all $i$ \cite[18,\textnormal{ Lemma} 2]{Mat}, the induction hypothesis implies that $\textnormal{Tor}^{\overline{R}}_{i}(\overline{M},\overline{N})=0$ for all $i\geq 1$. Therefore the result follows from $(4.1.1)$ and Nakayama's lemma.
\end{proof}

\begin{remark} It should be noted that Theorem $4.1$ follows from Dao's results in some important cases. (See \cite[3.5--4.1]{Da1}.) For example, if $(R,\mathfrak m)$ is a three-dimensional admissible hypersurface which is an isolated singularity (i.e., $R_{p}$ is a regular local ring for all $p\in X^{2}(R)$), then the conclusion of Theorem $4.1$ follows even without assumption $(1)$. This is because one of the modules considered will have dimension at most one. Another important result of Dao states that if $(R,\mathfrak m)$ is an equicharacteristic admissible hypersurface which is an isolated singularity, and if $M$ and $N$ are finitely generated $R$-modules such that $\textnormal{dim}(M)+\textnormal{dim}(N)\leq \textnormal{dim}(R)$, then $(M,N)$ is rigid, i.e., if $\textnormal{Tor}^{R}_{n}(M,N)=0$ for some $n\geq 1$, then $\textnormal{Tor}^{R}_{i}(M,N)=0$ for all $i\geq n$.

Dao also used Hochster's $\theta$ function \cite{Da1} to prove that there are admissible local hypersurfaces, which are isolated singularities, over which every module is rigid, e.g., two-dimensional hypersurfaces, four-dimensional equicharacteristic hypersurfaces, or one-dimensional hypersurface domains. This suggests that, if the ring $R$ in Theorem $4.1$ is such a hypersurface, then the Cohen-Macaulayness assumption on the modules $M$ and $N$ might rarely occur. We record the following observation for the special case where $M=N$.
\end{remark}

\begin{prop} Let $(R,\mathfrak m)$ be a local complete intersection, and let $M$ be a finitely generated $R$-module. Assume $\textnormal{Tor}^{R}_{i}(M,M)=0$ for all $i\geq 1$. If $M$ or $M\otimes_{R}M$ is Cohen-Macaulay, then $M$ is free.
\end{prop}

\begin{proof} Let $d=\textnormal{dim}(R)$. If $M=0$, then there is nothing to prove. So we may assume $M$ is nonzero. Suppose $M\otimes_{R}M$ is Cohen-Macaulay. Then, since $\textnormal{Tor}^{R}_{i}(M,M)=0$ for all $i\geq 1$, the depth formula of Theorem $2.7$ holds. This implies that $\textnormal{depth}_{R}(M\otimes_{R}M) \leq \textnormal{depth}_{R}(M)$ and hence $M$ is Cohen-Macaulay. Thus, to prove the claim, it suffices to assume $M$ is Cohen-Macaulay. We know, by \cite[1.2]{Jo2}, that $M$ has finite projective dimension. Therefore, if $d=0$, then $M$ is free by the Auslander-Buchsbaum equality. Hence we may assume $d\geq 1$. Furthermore, localizing the depth formula of Theorem $2.7$, we see that $M$ satisfies $(S_{1})$, i.e., $M$ is torsion-free. Thus $\textnormal{Ann}_{R}(M)$ is contained in the set of zero-divisors of $R$. This shows that $\textnormal{dim}_{R}(M)=d$, and hence $M$ is free.
\end{proof}

Note that the module $M$ in Proposition $4.3$ may not be free if $\textnormal{Tor}^{R}_{i}(M,M)$ does not vanish for all $i\geq 1$.

\begin{example} Let $k$ be a field, $R=k[[X,Y,Z]]/(XY)$ and $M=R/(z)$. Then $R$ is a two-dimensional hypersurface and $M$ is a Cohen-Macaulay $R$-module that has projective dimension one. Moreover, since $M$ is a non-zero cyclic $R$-module,  $\textnormal{Tor}^{R}_{1}(M,M)\neq 0$.
\end{example}

There also exist non-free modules $M$ such that $\textnormal{Tor}^{R}_{i}(M,M)=0$ for all $i\geq 1$. We record the folllowing example for further use.

\begin{example} Let $k$ be a field and put $R=k[[X,Y,W,Z]]/(XW-YZ)$. Then $R$ is a three-dimensional hypersurface domain. Let $M$ be the cokernel of the following map \cite[2.3]{HJ}:

$$\xymatrixcolsep{13pt}\xymatrix{R\ar[rrr]^-{\left[   \begin{array}{c} W\\ Y \\ X \\ Z
\end{array}  \right]} & &
& R^{(4)}} $$
Then $M$ has projective dimension one, and hence $\textnormal{depth}_{R}(M)=2$. Moreover, it can be checked that $\textnormal{Tor}^{R}_{1}(M,M)=0$. Therefore $\textnormal{Tor}^{R}_{i}(M,M)=0$ for all $i\geq 1$. It follows by Auslander's depth formula that $\textnormal{depth}_{R}(M\otimes_{R}M)=1$. As $\textnormal{dim}_{R}(M)=3$, both $M$ and $M\otimes_{R}M$ are not Cohen-Macaulay.
\end{example}

The next theorem follows by adapting the proof of \cite[6.5]{Da1}.

\begin{theorem} (H. Dao) Let $R=S/(\underline{f})$ where $(S,\mathfrak n)$ is a local ring and $\underline{f}=f_{1},f_{2}, \dots, f_{r}$, for $r\geq 1$, is a regular sequence of $S$. Let $M$ and $N$ be non-zero finitely generated $R$-modules. Assume the following conditions hold:
\begin{enumerate}
\item $M$ has finite projective dimension as an $S$-module.
\item $M\otimes_{R}N$ has finite length.
\item $\textnormal{depth}(M)+\textnormal{depth}(N)\geq \textnormal{depth}(S)$.
\end{enumerate}
Then $\textnormal{depth}(M)+\textnormal{depth}(N)= \textnormal{depth}(S)$; and $\textnormal{Tor}^{R}_{i}(M,N)\neq0$ if and only if $i$ is an even positive integer.
\end{theorem}

\begin{proof} Let $R_{j}=R_{j-1}/(f_{j})$ where $1\leq j \leq r$, $R_{0}=S$ and $R_{r}=R$. Assume $j$ is an integer such that $1\leq j \leq r$, and $\textnormal{Tor}^{R_{j-1}}_{i}(M,N)=0$ for all odd $i\geq 1$. The change of rings long exact sequence of Tor \cite[2.1]{HW1}, applied to $R_{j}$ and $R_{j-1}$, gives the following exact sequence for $i\geq 1$:
$$(4.3.1)\;\;  \textnormal{Tor}^{R_{j-1}}_{i}(M,N) \rightarrow \textnormal{Tor}^{R_{j}}_{i}(M,N) \rightarrow \textnormal{Tor}^{R_{j}}_{i-2}(M,N) \rightarrow \textnormal{Tor}^{R_{j-1}}_{i-1}(M,N)$$
Then we have that $\textnormal{Tor}^{R_{j}}_{i}(M,N)  \twoheadrightarrow \textnormal{Tor}^{R_{j}}_{i-2}(M,N)$ if $i$ is an even positive integer, and $\textnormal{Tor}^{R_{j}}_{i}(M,N)\hookrightarrow \textnormal{Tor}^{R_{j}}_{i-2}(M,N)$ if $i$ is odd. Since $\textnormal{Tor}^{R_{j}}_{-1}(M,N)=0$,  this shows that  $\textnormal{Tor}^{R_{j}}_{i}(M,N)\neq0$ if and only if $i$ is an even positive integer. Therefore, to prove the assertion concerning $\textnormal{Tor}^{R}_{i}(M,N)$, it is enough to prove that $\textnormal{Tor}^{S}_{i}(M,N)=0$ for all odd $i\geq 1$.

Let $q=\text{sup}\left\{i:\textnormal{Tor}^{S}_{i}(M,N)\neq 0\right\}$. Then, by $(1)$, $q<\infty$. Moreover, by $(2)$,   $\text{depth}\left(\textnormal{Tor}^{S}_{q}(M,N)\right)=0$. Then Auslander's depth formula implies that $\textnormal{depth}(M)+\textnormal{depth}(N)=\textnormal{depth}(S)+\textnormal{depth} \left(\textnormal{Tor}^{S}_{q}(M,N)\right)-q=\textnormal{depth}(S)-q$. Therefore, by $(3)$, we get $q\leq 0$, i.e., $q=0$. This shows that $\textnormal{depth}(M)+\textnormal{depth}(N)=\text{depth}(S)$ and $\textnormal{Tor}^{S}_{i}(M,N)=0$ for all $i\geq 1$. This completes the proof.
\end{proof}

We begin recording some corollaries of the previous theorem. The first one corroborates an example of Bergh and Jorgensen \cite{BJ}.

\begin{corollary} Let $(R,\mathfrak m)$ be a local complete intersection, which is not a field, of codimension $d$ and dimension $d$, and let $M$ and $N$ be maximal Cohen-Macaulay $R$-modules. Assume $M\otimes_{R}N$ has finite length. Then $\textnormal{Tor}^{R}_{i}(M,N)\neq 0$ if and only if $i$ is a nonnegative even integer.
\end{corollary}

Note that the New Intersection Theorem of Peskine and Szpiro \cite{PS}, Hochster \cite{H}, and P. Roberts \cite{R1}, \cite{R2} gives the inequality
$$\textnormal{dim}(N)\leq \textnormal{dim}(M\otimes_{R}N)+\textnormal{depth}(S)-\textnormal{depth}(M)=\textnormal{depth}(S)-\textnormal{depth}(M)$$ for the modules $M$ and $N$ of Theorem $4.6$. Thus the conclusion of Theorem $4.6$, concerning the depths of the modules considered, also follows from the above inequality. (Note also that the module $N$ is Cohen-Macaulay.)

A generalization of the Intersection Theorem, proved by T. Sharif and S. Yassemi \cite[3.1]{SY}, and \cite[5.11]{AGP} show that two finitely generated modules $X$ and $Y$ over a $d$-dimensional local complete intersection ring $R$ must satisfy the inequality
$$\textnormal{dim}_{R}(Y)+\textnormal{depth}_{R}(X)\leq \textnormal{dim}_{R}(X\otimes_{R}Y)+d+\textnormal{cx}_{R}(X)$$
The above inequality and Theorem $4.6$ now implies the following rigidity result.

\begin{prop} Let $(R,\mathfrak m)$ be a local ring such that $\hat{R}=S/(\underline{f})$ where $(S,\mathfrak{n})$ is a complete unramified regular local ring and $\underline{f}=f_{1},f_{2},\dots,f_{c}$, for $c\geq 1$, is a regular sequence of $S$ contained in $\mathfrak{n}^{2}$. Let $M$ and $N$ be non-zero finitely generated $R$-modules. Assume the following conditions hold:
\begin{enumerate}
\item $M$ and $N$ are Cohen-Macaulay.
\item $M\otimes_{R}N$ has finite length.
\item $\textnormal{Tor}^{R}_{n}(M,N)=\textnormal{Tor}^{R}_{n+1}(M,N)=\dots=\textnormal{Tor}^{R}_{n+c-1}(M,N)=0$ for some positive integer
$n$.
\item If $c=1$, assume further that $n$ is a positive even integer.
\end{enumerate}
Then $\textnormal{Tor}^{R}_{i}(M,N)=0$ for all $i\geq n$.
\end{prop}

\begin{proof} Without loss of generality, we may assume $R$ is complete. Let $d=\textnormal{dim}(R)$. If $d=0$, then Theorem $2.3$ implies that $\textnormal{Tor}^{R}_{i}(M,N)=0$ for all $i\geq n$. Therefore we may assume $d\geq 1$. We have the following equality that follows from \cite[3.1]{SY}:
$$(4.8.1)\;\; \textnormal{dim}_{R}(N)+\textnormal{depth}_{R}(M)\leq \textnormal{dim}_{R}(M\otimes_{R}N)+d+\textnormal{cx}_{R}(M)$$
Since $M$ is Cohen-Macaulay and $M\otimes_{R}N$ has finite length, it follows from $(4.8.1)$ that $\displaystyle{\textnormal{dim}_{R}(N)+\textnormal{dim}_{R}(M)\leq d+\textnormal{cx}_{R}(M)}$.
As $\textnormal{cx}_{R}(M)\leq c$, we have that \\ $\displaystyle{\textnormal{dim}_{R}(N)+\textnormal{dim}_{R}(M)\leq d+c}$. Now, if $\displaystyle{\textnormal{dim}_{R}(N)+\textnormal{dim}_{R}(M)=d+c}$, then $(1)$ and Theorem $4.6$ imply that $\textnormal{Tor}^{R}_{i}(M,N)\neq 0$ if and only if $i$ is a positive even integer.
This contradicts $(3)$ if $c\geq 2$, and $(4)$ if $c=1$. Thus $\displaystyle{\textnormal{dim}_{R}(N)+\textnormal{dim}_{R}(M)< d+c}$ so that Theorem $2.3$ gives the desired conclusion.
\end{proof}

We should note that, in Proposition $4.8$, one can replace $c$ with the maximum of the complexities of $M$ and $N$ by using
\cite[2.6]{Jo1}, provided $n>\textnormal{dim}(R)$. Finally we prove the following corollary which, in particular, explains why the module $M\otimes_{R}N$ is not Cohen-Macaulay in Example $3.11$.

\begin{prop} Let $(R,\mathfrak m)$ be a local ring such that $\hat{R}=S/(\underline{f})$ where $(S,\mathfrak{n})$ is a complete unramified regular local ring and $\underline{f}=f_{1},f_{2},\dots,f_{c}$, for $c\geq 1$, is a regular sequence of $S$ contained in $\mathfrak{n}^{2}$. Let $M$ and $N$ be non-zero finitely generated $R$-modules. Assume the following conditions hold:
\begin{enumerate}
\item $\textnormal{Tor}^{R}_{1}(M,N)=\textnormal{Tor}^{R}_{2}(M,N)=\ldots =\textnormal{Tor}^{R}_{c}(M,N)=0$.
\item $M$, $N$ and $M\otimes_{R}N$ are Cohen-Macaulay.
\item If $c=1$, assume further that $\textnormal{dim}(M)+\textnormal{dim}(N)\leq d$.
\end{enumerate}
Then  $\textnormal{Tor}^{R}_{i}(M,N)=0$ for all $i\geq 1$. In particular the depth formula holds for $M$ and $N$.
\end{prop}

\begin{proof} Without loss of generality, we may assume $R$ is complete. If $c=1$, then the result follows from Theorem $4.1$.

Assume now that $M$, $N$ and $M\otimes_{R}N$ are Cohen-Macaulay, $\textnormal{Tor}^{R}_{1}(M,N)=\textnormal{Tor}^{R}_{2}(M,N)=\ldots =\textnormal{Tor}^{R}_{c}(M,N)=0$, and $c \geq 2$. We proceed by induction on $d=\textnormal{dim}(R)$ to prove that $\textnormal{Tor}^{R}_{i}(M,N)=0$ for all $i\geq 1$. If $d=0$, then the result follows from Theorem $2.3$. Therefore we may assume $d\geq 1$. Now let $p$ be a non-maximal prime ideal of $R$. If $R_{p}$ has codimension $c$, then the induction hypothesis, applied to $M$ and $N$, implies that
$\textnormal{Tor}^{R_{p}}_{i}(M_{p},N_{p})=0$ for all $i\geq 1$. If, on the other hand, $R_{p}$ has codimension less than $c$, Theorem $2.2$ and $(1)$ imply that $\textnormal{Tor}^{R_{p}}_{i}(M_{p},N_{p})=0$ for all $i\geq 1$. This shows that $\textnormal{Tor}^{R}_{i}(M,N)$ has finite length for all $i\geq 1$. If $\textnormal{depth}(M\otimes_{R}N)=0$, then $M\otimes_{R}N$ has finite length so that the result follows from Proposition $4.8$. Thus we may assume $\textnormal{depth}(M\otimes_{R}N)>0$. In particular $\textnormal{depth}(M)>0$ and $\textnormal{depth}(N)>0$. In this case the result concerning the vanishing of  $\textnormal{Tor}^{R}_{i}(M,N)$ follows from Theorem $2.8$. Therefore, by Theorem $2.7$, the depth formula holds for $M$ and $N$.
\end{proof}

\begin{remark}
Jorgensen \cite{Jo1} asks whether there exist a local complete intersection $R$ of codimension $c\geq 2$, and finitely generated $R$-modules $M$ and $N$ such that $M\otimes_{R}N$ has finite length, $\textnormal{Tor}^{R}_{i}(M,N)=\textnormal{Tor}^{R}_{i+1}(M,N)=\dots=\textnormal{Tor}^{R}_{i+n-1}(M,N)=0$ and $\textnormal{Tor}^{R}_{i+n}(M,N)\neq 0$ for some $i\geq 1$ and $n\geq 2$. Therefore it would be interesting to know whether one could remove the Cohen-Macaulayness assumption on the modules $M$ and $N$ in Proposition $4.8$ or $4.9$.
\end{remark}

Next our aim is to prove Theorem $1.2$, advertised in the introduction. We will give a proof after several preliminary results.

Let $(R,\mathfrak m)$ be a $d$-dimensional local complete intersection, and let $M$ and $N$ be finitely generated $R$-modules. Set  $b=\textnormal{max}\{\textnormal{depth}_{R}(M),\textnormal{depth}_{R}(N)\}$ and assume that $\textnormal{Tor}^{R}_{n}(M,N)=\textnormal{Tor}^{R}_{n+1}(M,N)=\ldots =\textnormal{Tor}^{R}_{n+\textnormal{cx}_{R}(M)-1}(M,N)=0$ for some $n\geq d-b+1$. Then, if $N$ has finite length, it is proved in \cite[3.3.ii]{Be} that $\textnormal{Tor}^{R}_{i}(M,N)=0$ for all $i\geq d-\textnormal{depth}_{R}(M)+1$ (cf. also the proof of \cite[2.6]{Jo1}.) The next proposition is similar to \cite[3.3.ii]{Be} and Proposition $3.7$, except we assume $\textnormal{Tor}^{R}_{i}(M,N)$ has finite length for certain values of $i$, rather than assuming that $N$ has finite length.

\begin{prop} Let $(R,\mathfrak m)$ be a $d$-dimensional local complete intersection, and let $M$ and $N$ be finitely generated $R$-modules. Set $b=\textnormal{max}\{\textnormal{depth}(M),\textnormal{depth}(N)\}$ and $r=\textnormal{min}\{\text{cx}_{R}(M),\text{cx}_{R}(N)\}$. Let $n$ and $w$ be integers such that $w\geq 0$ and $n\geq d-b+1$. Assume the following conditions hold:
\begin{enumerate}
\item $\textnormal{Tor}^{R}_{n}(M,N)=\textnormal{Tor}^{R}_{n+1}(M,N)=\ldots =\textnormal{Tor}^{R}_{n+r-1}(M,N)=0$.
\item $\textnormal{Tor}^{R}_{n+2w+i}(M,N)$ has finite length for all $i=1,2,\dots, r$.
\end{enumerate}
Then either $\textnormal{Tor}^{R}_{i}(M,N)=0$ for all $i\geq d-b+1$, or $\textnormal{depth}(\textnormal{Tor}^{R}_{n-1}(M,N))=0$.
\end{prop}

\begin{proof} Without loss of generality we may assume $r=\textnormal{cx}_{R}(M)$ and $R$ is complete with an infinite residue field. If $r=0$, then Theorem $2.4$ implies that $\textnormal{Tor}^{R}_{i}(M,N)=0$ for all $i\geq d-b+1$. Suppose now $r\geq 1$ and $\textnormal{depth}(\textnormal{Tor}^{R}_{n-1}(M,N))\neq 0$. We will show that $\textnormal{Tor}^{R}_{i}(M,N)=0$ for all $i\geq d-b+1$. We may choose, using Theorem $2.5$, a complete intersection $S$ and a non-zerodivisor $f$ of $S$ such that $R=S/(f)$ and $\textnormal{cx}_{S}(M)=r-1$. The change of rings long exact sequence of Tor \cite[2.1]{HW1}, applied to $S$ and $R$, gives the following exact sequence
\vspace{0.1in}

$(4.11.1)\;\;\; \textnormal{Tor}^{R}_{j+2}(M,N) \rightarrow \textnormal{Tor}^{R}_{j}(M,N) \rightarrow \textnormal{Tor}^{S}_{j+1}(M,N) \rightarrow \textnormal{Tor}^{R}_{j+1}(M,N)\rightarrow \\ \textnormal{Tor}^{R}_{j-1}(M,N)$ \; for all $j\geq 0$.
\vspace{0.1in}

\noindent Assume $r=1$. Then $M$ has finite projective dimension over $S$. Thus Theorem $2.4$ implies that $\textnormal{Tor}^{S}_{i}(M,N)=0$ for all $i\geq d-b+2$, and hence $(4.11.1)$ gives the injection $\textnormal{Tor}^{R}_{n+1}(M,N)\hookrightarrow \textnormal{Tor}^{R}_{n-1}(M,N)$. Note that, as $\textnormal{Tor}^{R}_{i}(M,N)\cong \textnormal{Tor}^{R}_{i+2}(M,N)$ for all $i\geq d-b+1$, $\textnormal{Tor}^{R}_{n+1}(M,N)$ has finite length by $(2)$. Therefore $\textnormal{Tor}^{R}_{n+1}(M,N)=0$, and so the result follows from $(1)$ and Theorem $2.4$.

Assume now $r\geq 2$. Then $(4.11.1)$ shows that $\textnormal{Tor}^{S}_{n+1+2w+i}(M,N)$ has finite length for all $i=1, 2,\dots, r-1$, $\textnormal{Tor}^{S}_{n+1}(M,N)= \ldots =\textnormal{Tor}^{S}_{n+r-1}(M,N)=0$ and $\textnormal{Tor}^{R}_{n-1}(M,N)\cong \textnormal{Tor}^{S}_{n}(M,N)$. Thus the induction hypothesis implies that $\textnormal{Tor}^{S}_{i}(M,N)=0$ for all $i\geq d-b+2$, and hence $\textnormal{Tor}^{R}_{i}(M,N)\cong \textnormal{Tor}^{R}_{i+2}(M,N)$ for all $i\geq d-b+1$. Now the claim follows as in the case where $r=1$.
\end{proof}

The following corollary of Proposition $4.11$ shows that, in Theorem $2.8$, one can replace the codimension of the ring with the minimum of the complexities of the modules, provided one of the modules considered is maximal Cohen-Macaulay.

\begin{corollary} Let $(R,\mathfrak m)$ be a local complete intersection, and
let $M$ and $N$ be finitely generated $R$-modules. Set $r=\textnormal{min}\{\text{cx}_{R}(M),\text{cx}_{R}(N)\}$. Assume the following conditions hold:
\begin{enumerate}
\item $M$ is maximal Cohen-Macaulay.
\item $\textnormal{depth}_{R}(M\otimes_{R}N)>0$.
\item $\textnormal{Tor}^{R}_{1}(M,N)=\textnormal{Tor}^{R}_{2}(M,N)=\ldots =\textnormal{Tor}^{R}_{r}(M,N)=0$.
\item $\textnormal{Tor}^{R}_{i}(M,N)$ has finite length for all $i\geq 1$.
\end{enumerate}
Then $\textnormal{Tor}^{R}_{i}(M,N)=0$ for all $i\geq 1$.
\end{corollary}

Note that Example $3.11$ and the example given preceding Theorem $3.4$ show, respectively,
 that assumptions $(2)$ and $(3)$ of Corollary $4.12$ cannot be removed.
 Furthermore, if $R$, $M$ and $N$ are as in Example $3.14$, and $M'=\textnormal{syz}^{R}_{1}(M)$, then
 $\textnormal{Tor}^{R}_{1}(M',N)=0$, and $M'\otimes_{R}N$ and $\textnormal{Tor}^{R}_{2}(M',N)$ have positive depth. This shows that one cannot remove assumption $(4)$ in Corollary $4.12$.

\begin{prop} Let $(R,\mathfrak m)$ be a local complete intersection, and let $M$ and $N$ be finitely generated $R$-modules. Set $r=\textnormal{min}\{\text{cx}_{R}(M),\text{cx}_{R}(N)\}$. Assume the following conditions hold:
\begin{enumerate}
\item $\textnormal{Tor}^{R}_{1}(M,N)=\textnormal{Tor}^{R}_{2}(M,N)=\ldots =\textnormal{Tor}^{R}_{r-1}(M,N)=0$.
\item $M$ is maximal Cohen-Macaulay.
\item $M\otimes_{R}N$ is reflexive.
\item $N$ is torsion-free.
\item $\textnormal{Tor}^{R}_{i}(M,N)_{q}=0$ for all $q\in X^{1}(R)$ and $i\geq 1$.
\end{enumerate}
Then $\textnormal{Tor}^{R}_{i}(M,N)=0$ for all $i\geq 1$.
\end{prop}

\begin{proof} Since $M$ is maximal Cohen-Macaulay, we may assume $r\geq 1$ (see Theorem $3.8$.)
Let $d=\text{dim}(R)$. We proceed by induction on $d$. If $d\leq 1$, then $\textnormal{Tor}^{R}_{i}(M,N)=0$ for all
$i\geq 1$ by $(5)$. Assume now $d\geq 2$. Then the induction hypothesis implies that
$\textnormal{Tor}^{R}_{i}(M,N)$ has finite length for all $i\geq 1$. Consider the pushforward of $M$:
$$(4.13.1)\;\; 0 \rightarrow M \rightarrow R^{(m)} \rightarrow M_{1} \rightarrow 0$$
By Proposition $3.2(3)$, $\textnormal{Tor}^{R}_{1}(M_{1},N)=0$ and
$M_{1}\otimes_{R}N$ is torsion-free. Moreover $(1)$ and $(4.13.1)$ show
that $\textnormal{Tor}^{R}_{1}(M_{1},N)=\textnormal{Tor}^{R}_{2}(M_{1},N)=\ldots
=\textnormal{Tor}^{R}_{r}(M_{1},N)=0$. Now, since $\textnormal{Tor}^{R}_{i}(M_{1},N)$ has finite length for all $i\geq
1$ and $M_{1}$ is maximal Cohen-Macaulay (see Proposition $3.1(1)$), Corollary $4.12$ implies that $\textnormal{Tor}^{R}_{i}(M_{1},N)=0$ for all $i\geq
1$. Thus the result follows from $(4.13.1)$.
\end{proof}

\begin{prop} Let $(R,\mathfrak m)$ be a one-dimensional local complete intersection, and
let $M$ and $N$ be finitely generated $R$-modules, at least one of which has constant rank.
Set $r=\textnormal{max}\{\text{cx}_{R}(M),\text{cx}_{R}(N)\}$. Assume the following conditions hold:
\begin{enumerate}
\item $\textnormal{Tor}^{R}_{1}(M,N)=\textnormal{Tor}^{R}_{2}(M,N)=\ldots =\textnormal{Tor}^{R}_{r-1}(M,N)=0$.
\item $M$ and $M\otimes_{R}N$ is torsion-free.
\end{enumerate}
Then $\textnormal{Tor}^{R}_{i}(M,N)=0$ for all $i\geq 1$.
\end{prop}

\begin{proof} Note that, as $M$ is torsion-free, we may assume $r\geq 1$. Suppose now $N$ has torsion, i.e., $\textnormal{depth}(N)=0$. Then we can choose a maximal Cohen-Macaulay approximation for $N$ \cite{AuB}, i.e., we have an exact sequence
$$(4.14.1)\;\; 0 \rightarrow P \rightarrow X \rightarrow N \rightarrow 0$$ where $X$ is maximal Cohen-Macaulay and
$P$ has finite injective dimension. As $R$ is Gorenstein, $P$ has
also finite projective dimension. Since $R$ has dimension one and
$\textnormal{depth}_{R}(N)=0$, depth lemma implies that $P$ is free.
In particular, if $N$ has constant rank so does $X$. Moreover,
tensoring $(4.14.1)$ with $M$, we have the following exact sequence
$$(4.14.2)\;\; \textnormal{Tor}^{R}_{1}(M,N) \stackrel{\alpha}{\rightarrow} P\otimes_{R}M \rightarrow X\otimes_{R}M \rightarrow M\otimes_{R}N \rightarrow 0$$
Since $\textnormal{Tor}^{R}_{1}(M,N)$ is torsion and $M$ is torsion-free, $\alpha=0$ and hence $(4.14.2)$ yields the following exact sequence
$$(4.14.3)\;\; 0  \rightarrow P\otimes_{R}M \rightarrow X\otimes_{R}M \rightarrow M\otimes_{R}N \rightarrow 0$$
Now the depth lemma implies that $X\otimes_{R}M$ is torsion-free.
Furthermore, by $(4.14.1)$ and $(1)$, we have
$$(4.14.4)\;\; \textnormal{Tor}^{R}_{1}(M,X)=\textnormal{Tor}^{R}_{2}(M,X)=\ldots =\textnormal{Tor}^{R}_{r-1}(M,X)=0$$
Hence, replacing $N$ by $X$, we may assume that $N$ is torsion-free. Now, without loss of generality, we may assume that $N$
has constant rank. Then \cite[1.3]{HW1} gives the following short exact sequence
$$(4.14.5)\;\; 0\rightarrow N \rightarrow R^{(t)} \rightarrow C \rightarrow 0$$ where $C$ has finite length. Since $M\otimes_{R}N$ is
torsion-free, tensoring $(4.14.5)$ with $M$, we see that
$\textnormal{Tor}^{R}_{1}(C,M)=0$. Now, if $r=1$, \cite[3.3.ii]{Be}
implies that $\textnormal{Tor}^{R}_{i}(C,M)=0$ for all $i\geq 1$.
If, on the other hand, $r>1$, we use $(1)$ to deduce that
$$(4.14.6)\;\;\textnormal{Tor}^{R}_{1}(C,M)=\textnormal{Tor}^{R}_{2}(C,M)=\ldots =\textnormal{Tor}^{R}_{r}(C,M)=0.$$
Then, using \cite[3.3.ii]{Be} again with $(4.14.6)$, we have that
$\textnormal{Tor}^{R}_{i}(C,M)=0$ for all $i\geq 1$. Now the
conclusion follows from $(4.14.5)$.
\end{proof}

Theorem $1.2$ is now a special case of the next result:

\begin{theorem} Let $(R,\mathfrak m)$ be a local complete intersection, and let $M$ and $N$ be finitely generated $R$-modules, at least one of which has constant rank. Set $r=\textnormal{max}\{\text{cx}_{R}(M),\text{cx}_{R}(N)\}$. Assume the following conditions hold:
\begin{enumerate}
\item $\textnormal{Tor}^{R}_{1}(M,N)=\textnormal{Tor}^{R}_{2}(M,N)=\ldots =\textnormal{Tor}^{R}_{r-1}(M,N)=0$.
\item $M$ is maximal Cohen-Macaulay.
\item $M\otimes_{R}N$ is reflexive.
\item $N$ is torsion-free.
\end{enumerate}
Then $\textnormal{Tor}^{R}_{i}(M,N)=0$ for all $i\geq 1$.
\end{theorem}

\begin{proof} Let $d=\text{dim}(R)$. Since either $M$ or $N$ has constant rank, we may assume $d\geq 1$.
We now proceed by induction on $d$. If $d=1$, then the result follows from Proposition $4.14$.
So assume $d\geq 2$. In this case, the induction hypothesis and Proposition $4.13$ implies that $\textnormal{Tor}^{R}_{i}(M,N)=0$ for all $i\geq 1$.
\end{proof}

Our motivation for Theorem $4.15$ comes from the following question of Huneke and Wiegand \cite[page 473]{HW1}:

\begin{question}
If $(R,\mathfrak m)$ is a one-dimensional Gorenstein domain, and $M$ is a torsion-free $R$-module such that $M\otimes_{R}M^{\ast}$ is torsion-free, then is $M$ free?
\end{question}

In their remarkable paper, Huneke and Wiegand \cite[3.1]{HW1} proved that over a local hypersurface, if the tensor product of two modules, at least one of which has constant rank, is maximal Cohen-Macaulay, then one of them must be free. Using this result, they showed that \cite[5.2]{HW1} Question $4.16$ has a positive answer over any domain $R$ satisfying $(S_{2})$ (not necessarily Gorenstein and one-dimensional) provided $R_{p}$ is a hypersurface for all $p\in X^{1}(R)$. However, if the ring is not assumed to be a hypersurface (in codimension one), it is not known (at least to the author) whether Question $4.16$ has an affirmative answer, even over a complete intersection domain of codimension two. Following the same induction argument as in \cite[5.2]{HW1}, we will now establish a consequence of a theorem of Avramov--Buchweitz \cite[4.2]{AB} and Huneke--Jorgensen \cite[5.9]{HJ}.

\begin{prop} Let $(R,\mathfrak m)$ be a local complete intersection, and let $M$ be a finitely generated torsion-free $R$-module such that $M\otimes_{R}M^{\ast}$ is reflexive. For all $q\in X^{1}(R)$, assume one of the following holds:
\begin{enumerate}
\item $\text{Tor}_{i}^{R}(M,M^{\ast})_{q}=0$ for some even $i\geq 2$ and $\text{Tor}_{j}^{R}(M,M^{\ast})_{q}=0$ for some odd $j\geq 1$.
\item $M_{q}$ has constant rank and $\text{cx}_{R_{q}}(M_{q})\leq 1$, i.e., $M_{q}$ has bounded Betti numbers.
\end{enumerate}
Then $M$ is free.
\end{prop}

\begin{proof} We proceed by induction on $d=\text{dim}(R)$. First assume $(1)$. If $d=0$,
then $\textnormal{Ext}^{i}_{R}(M,M)$ is the Matlis dual of $\text{Tor}_{i}^{R}(M,M^{\ast})$.
So, by \cite[4.2]{AB}, $M$ has finite projective dimension, i.e., $M$ is free. Suppose now $d=1$. Consider the exact sequence
$$(4.17.1)\;\; 0 \rightarrow \textnormal{syz}^{R}_{j}(M) \rightarrow F \rightarrow \textnormal{syz}^{R}_{j-1}(M)\rightarrow 0$$
where $F$ is a free $R$-module. Tensoring $(4.17.1)$ by $M^{\ast}$ we get the exact sequence
$$(4.17.2)\;\; 0 \rightarrow \textnormal{Tor}^{R}_{1}(\textnormal{syz}^{R}_{j-1}(M),M^{\ast}) \rightarrow \textnormal{syz}^{R}_{j}(M)\otimes_{R}M^{\ast} \rightarrow F\otimes M^{\ast}$$
Since $\textnormal{Tor}^{R}_{1}(\textnormal{syz}^{R}_{j-1}(M),M^{\ast})\cong \text{Tor}_{j}^{R}(M,M^{\ast})=0$ and $M^{\ast}$ is torsion-free,
the depth lemma implies that $\textnormal{depth}_{R}(\textnormal{syz}^{R}_{j}(M)\otimes_{R}M^{\ast})>0$. As $R$ has dimension one, it follows
that $\textnormal{syz}^{R}_{j}(M)\otimes_{R}M^{\ast}$ is torsion-free. Since $\textnormal{Ext}^{1}_{R}(\textnormal{syz}^{R}_{j}(M),M)$ has finite length,
\cite[5.9]{HJ} implies that $\textnormal{Ext}^{1}_{R}(\textnormal{syz}^{R}_{j}(M),M)=0$, i.e., $\textnormal{Ext}^{j+1}_{R}(M,M)=0$. As $j$ is odd,
using \cite[4.2]{AB} one more time, we conclude that $M$ is free.

Next assume $(2)$. If $d=0$, then the result follows by assumption.
So suppose $d=1$.  If $\text{cx}_{R}(M)=1$, then by \cite[1.3]{Jo2},
$\text{cx}_{R}(M^{\ast})=1$ and $\text{Tor}_{i}^{R}(M,M^{\ast})\neq
0$ for infinitely many $i$. This contradicts Proposition $4.14$.
Therefore $\text{cx}_{R}(M)=0$, and hence $M$ is free. This proves
the proposition for the case where $d\leq 1$.

Suppose now $d\geq 2$. Then the induction hypothesis implies that $M$ is free on $X^{1}(R)$. In this case it is proved in \cite[5.2]{HW1} that $M$ is free (cf. also the proof of \cite[3.3]{Au}.) Here we include the proof for completeness. It is known that \cite[A.1]{AuG} the map $\alpha_{M}: M\otimes_{R}M^{\ast} \rightarrow \text{Hom}_{R}(M,M)$ given by $\alpha_{M}(a \otimes f)(x)=f(x)a$ for all $a$ and $x$ in $M$ and $f$ in $M^{\ast}$ is surjective if and only if $M$ is free. Consider the exact sequence
$$(4.17.3)\;\; 0 \rightarrow B \rightarrow M\otimes_{R}M^{\ast} \stackrel{\alpha_{M}}{\longrightarrow} \text{Hom}_{R}(M,M) \rightarrow C \rightarrow 0$$
Note that, since $M$ is free on $X^{1}(R)$, $(\alpha_{M})_{q}$ is an
isomorphism for all $q\in X^{1}(R)$. Therefore, since
$M\otimes_{R}M^{\ast}$ is torsion-free, $B=0$. So, if $C \neq 0$,
localizing $(4.17.3)$ at an associated prime ideal of $C$, we see
that the depth lemma gives a contradiction. Thus $C=0$ and hence $M$
is free.
\end{proof}

Considering Proposition $4.17$, it seems reasonable to ask the following weaker form of Question $4.16$ for complete intersections:

\begin{question} Let $(R,\mathfrak m)$ be a one-dimensional local complete intersection domain, and let $M$ be a finitely generated $R$-module such that $M$ and $M\otimes_{R}M^{\ast}$ are torsion-free. If $\text{Tor}_{i}^{R}(M,M^{\ast})=0$ for some $i\geq 1$, then is $M$ free?
\end{question}

Next we include an example which shows that Proposition $4.17$ does not hold in general if the tensor product of the modules considered is not reflexive.

\begin{example} Let $R$ and $M$ be as in Example $4.5$. Then $M\otimes_{R}M^{\ast}$ is not reflexive.
We show this as an application of the Auslander's depth formula (cf. also \cite[3.3]{Au}).
Recall that $R$ is a three-dimensional hypersurface domain and $M$ is a torsion-free $R$-module that has projective dimension one.
Consider the following exact sequence:
$$(4.19.1)\;\; 0\rightarrow R \rightarrow R^{(4)} \rightarrow M \rightarrow 0$$
Note that $M^{\ast}$ is non-zero as $M$ is torsion-free. Now,
tensoring $(4.19.1)$ with $M^{\ast}$, we see that
$\text{Tor}_{1}^{R}(M,M^{\ast})=0$. Thus the depth formula holds for
$M$ and $M^{\ast}$. Suppose now $M\otimes_{R}M^{\ast}$ is reflexive.
Since $\textnormal{depth}_{R}(M)=2$, the depth formula implies that
$M^{\ast}$ is maximal Cohen-Macaulay. Now let $p$ be a prime ideal
$R$ such that $\textnormal{dim}(R_{p})=2$ and $M_{p}\neq 0$. As
$M^{\ast}_{p}$ is non-zero, localizing the depth formula at $p$, we
conclude that $M_{p}$ is maximal Cohen-Macaulay. Thus $M$ satisfies
$(S_{2})$, i.e., $M$ is reflexive. Since $M^{\ast}$ is maximal
Cohen-Macaulay, so is $M^{\ast\ast}$, and this gives the required
contradiction.

\end{example}

We finish this section with two more applications of the pushforward.
Our results will slightly improve upon two of the main theorems of \cite{HJW}.

Suppose $(R,\mathfrak m)$ is a local complete intersection of codimension $c$, and $M$ and $N$ are finitely generated $R$-modules.
If $R$ is a regular local ring, i.e., if $c=0$ and $M\otimes_{R}N$ is torsion-free,
then it follows from \cite[\text{Corollary} 2]{Li} that $\textnormal{Tor}^{R}_{i}(M,N)=0$
for all $i\geq 1$. Moreover, if $R$ is a hypersurface, i.e., if $c=1$, and $M\otimes_{R}N$ is reflexive, then
\cite[2.7]{HW1} shows that $\textnormal{Tor}^{R}_{i}(M,N)=0$ for all $i\geq 1$, provided either $M$ or $N$ has constant rank.
Now, assuming $c\geq 2$, we will see that vanishing of $c-1$
consecutive $\textnormal{Tor}^{R}_{i}(M,N)$ will give a similar rigidity result (cf. also \cite[2.1(i)]{Mu}).

\begin{prop} Let $(R,\mathfrak m)$ be a local complete intersection of codimension $c$, and let $M$ and $N$ be finitely generated $R$-modules, at least one of which has constant rank. Assume the following conditions hold:
\begin{enumerate}
\item $\textnormal{Tor}^{R}_{1}(M,N)=\textnormal{Tor}^{R}_{2}(M,N)= \ldots =\textnormal{Tor}^{R}_{c-1}(M,N)=0$.
\item $M\otimes_{R}N$ is reflexive.
\item If $c\geq 2$, assume further that $M$ and $N$ are torsion-free.
\end{enumerate}
Then $\textnormal{Tor}^{R}_{i}(M,N)=0$ for all $i\geq 1$.
\end{prop}

\begin{proof} By \cite[\text{Corollary} 2]{Li} and \cite[2.7]{HW1}, we may assume $c\geq 2$. Moreover, without loss of generality, we may assume $N$ has constant rank. Let $d=\text{dim}(R)$.
Now, if $d=0$, then $N$ is free so there is nothing to prove.
Assume $d=1$. In that case the result follows from Proposition
$4.14$ since both $M$ and $M\otimes_{R}N$ are torsion-free (Recall that
$\textnormal{cx}_{R}(M)\leq c$.) Hence suppose $d\geq 2$, and
consider the pushforward of $M$:
$$(4.20.1)\;\; 0\rightarrow M \rightarrow R^{(m)} \rightarrow M_{1} \rightarrow 0$$
Tensoring $(4.20.1)$ with $N$ and using $(1)$, we have
$$(4.20.2)\;\;\textnormal{Tor}^{R}_{1}(M_{1},N)=\textnormal{Tor}^{R}_{2}(M_{1},N)=\ldots =\textnormal{Tor}^{R}_{c}(M_{1},N)=0$$
Moreover the induction hypothesis and Proposition $3.2(3b)$ show
that $M_{1}\otimes_{R}N$ is torsion-free (If $R_{p}$ has codimension less than $c$ for a prime ideal $p$ of $R$, then we use
$(4.20.2)$ and Theorem $2.2$.) As $N$ has constant rank, \cite[1.3]{HW1} gives the following exact sequence
$$(4.20.3)\;\; 0\rightarrow N \rightarrow R^{(t)} \rightarrow C \rightarrow 0$$ where $C$ is torsion.
It now follows from $(4.20.3)$ that $\textnormal{Tor}^{R}_{1}(C,M_{1})=0$. Since $\textnormal{Tor}^{R}_{1}(M_{1},N)=0$, we have that
$$(4.20.4)\;\;\textnormal{Tor}^{R}_{1}(C,M_{1})=
\textnormal{Tor}^{R}_{2}(C,M_{1})=\ldots =\textnormal{Tor}^{R}_{c+1}(C,M_{1})=0$$
Now Theorem $2.2$ implies that $\textnormal{Tor}^{R}_{i}(C,M_{1})=0$ for all $i\geq 1$, and hence $\textnormal{Tor}^{R}_{i}(M,N)=0$ for all $i\geq 1$.
\end{proof}

For our last result, Proposition $4.22$, we need the following
theorem.

\begin{theorem} (\cite{HJW}) Let $(R,\mathfrak m)$ be a two-dimensional local complete intersection of codimension $c\geq 1$,
and let $M$ and $N$ be finitely generated $R$-modules. Assume the following conditions hold:
\begin{enumerate}
\item $\textnormal{Tor}^{R}_{n}(M,N)=\textnormal{Tor}^{R}_{n+1}(M,N)=\ldots =\textnormal{Tor}^{R}_{n+c-1}(M,N)=0$ for some positive integer $n$.
\item $M$ and $N$ are torsion-free.
\item $N$ has constant rank.
\item $M$ is free of constant rank on $X^{1}(R)$.
\end{enumerate}
Then $\textnormal{Tor}^{R}_{i}(M,N)=0$ for all $i\geq n$.
\end{theorem}

\begin{proof} On reading through the proof of \cite[2.2]{HJW}, we see that the conclusion of the theorem does not change if
$c$ is any positive integer (which is assumed to be two in the
proof), provided one assumes $(1)$ and uses Theorem $2.3$. Therefore
the desired result follows from the proof of \cite[2.2]{HJW}.
\end{proof}

Huneke-Jorgensen-Wiegand \cite[2.4]{HJW} proved that, over a local
complete intersection $R$ of codimension two, if $M$ and $N$ are
finitely generated reflexive $R$-modules such that $N$ has constant
rank, $M$ is free of constant rank on $X^{1}(R)$ and
$M\otimes_{R}N$ satisfies $(S_{3})$, then
$\textnormal{Tor}^{R}_{i}(M,N)=0$ for all $i\geq 1$. We observed in
Theorem $4.21$ that if $R$ has dimension two and codimension $c\geq
1$, then such modules $M$ and $N$ are $c$-rigid, even if $c\neq 2$.
This additional information now enables us to prove the next proposition.

\begin{prop} Let $(R,\mathfrak m)$ be a local complete intersection of codimension $c$, and let $M$ and $N$ be finitely generated $R$-modules. Assume the following conditions hold:
\begin{enumerate}
\item $\textnormal{Tor}^{R}_{1}(M,N)=\textnormal{Tor}^{R}_{2}(M,N)=\ldots =\textnormal{Tor}^{R}_{c-2}(M,N)=0$.
\item $M\otimes_{R}N$ satisfies $(S_{3})$.
\item $M$ and $N$ are reflexive.
\item $N$ has constant rank.
\item $M$ is free of constant rank on $X^{1}(R)$.
\end{enumerate}
Then $\textnormal{Tor}^{R}_{i}(M,N)=0$ for all $i\geq 1$.
\end{prop}

\begin{proof} The case where $c=0$ and $c=1$ follows from  \cite[\text{Corollary} 2]{Li} and \cite[2.7]{HW1}, respectively. Moreover, if $c=2$, then \cite[2.4]{HJW} implies that $\textnormal{Tor}^{R}_{i}(M,N)=0$ for all $i\geq 1$. Therefore we may assume $c\geq 3$. We have, by $(1)$ and \cite[2.1]{HJW}, that
$$(4.22.1)\;\;\textnormal{Tor}^{R}_{1}(M_{2},N)=\textnormal{Tor}^{R}_{2}(M_{2},N)=\ldots =\textnormal{Tor}^{R}_{c}(M_{2},N)=0$$
where $0\rightarrow M_{j-1} \rightarrow G_{j} \rightarrow M_{j}
\rightarrow 0$ is the pushforward of $M_{j}$ for $j=1,2$
($M_{0}=M$). We now proceed by induction on $d=\text{dim}(R)$. We
may assume $d\geq 2$ by $(5)$. Assume now $d=2$. Note that, by Proposition $3.1(1)$ and $(5)$, $M_{2}$ is maximal Cohen-Macaulay and is free of constant rank on $X^{1}(R)$.
Therefore Theorem $4.21$ and $(4.22.1)$ imply that
$\textnormal{Tor}^{R}_{i}(M_{2},N)=0$ for all $i\geq 1$, and hence $\textnormal{Tor}^{R}_{i}(M,N)=0$ for all $i\geq 1$.
Suppose $d\geq 3$. Since $\textnormal{depth}_{R}(M\otimes_{R}N)\geq 3$, using the induction hypothesis, we see that $M_{1}\otimes_{R}N$ satisfies $(S_{2})$, i.e., $M_{1}\otimes_{R}N$ is reflexive (cf. the proof of \cite[2.1]{HJW}.) In this case we use Proposition $4.20$ with $M_{1}$ and $N$ to conclude that $\textnormal{Tor}^{R}_{i}(M,N)=0$ for all $i\geq 1$.
\end{proof}

\section*{Acknowledgments}
I am grateful to my thesis advisors Roger Wiegand and Mark Walker
for their assistance in the preparation of this article. I also
would like to thank to the referee for his/her comments and
suggestions which have greatly improved the presentation of this
paper.

\end{document}